\documentclass[12pt, fleqn, oneside, reqno]{amsart}
\usepackage{newtxtext}
\usepackage[frenchmath]{newtxmath}
\usepackage[cal=boondoxo]{mathalfa}
\usepackage{mathrsfs}
\usepackage{tikz-cd}
\usetikzlibrary{decorations.pathmorphing,calc,arrows}
\usepackage{microtype}
\usepackage{hyperref}
\usepackage{nccmath}
\usepackage{setspace}

\allowdisplaybreaks

\graphicspath{ {assets/} }

\usepackage{pdflscape}

\usepackage[letterpaper,portrait,textwidth=5.75in]{geometry}

\newcommand{\CC}{\mathbf{C}}
\newcommand{\ZZ}{\mathbf{Z}}
\newcommand{\QQ}{\mathbf{Q}}

\newcommand{\PP}{\mathbf{P}}
\newcommand{\GG}{\mathbf{G}}
\newcommand{\DD}{\mathbf{D}}
\newcommand{\A}{\mathbf{A}}
\newcommand{\vep}{\varepsilon}
\newcommand{\Lotimes}{\otimes^{L}}

\newcommand{\sep}[1]{{#1}^{\mathrm{sep}}}
\newcommand{\alg}[1]{{#1}^{\mathrm{alg}}}

\usepackage{enumitem,moreenum}
\usepackage{amsmath}
\usepackage{amsthm}
\usepackage{mathtools}
\usepackage[utf8]{inputenc}
\usepackage[T1]{fontenc}
\usepackage{accents}

\tikzcdset{arrow style=tikz, diagrams={>=stealth'}}

\makeatletter
\renewenvironment{proof}[1][\proofname]{\par
  \normalfont
  \topsep6\p@\@plus6\p@ \trivlist
  \item[\hskip\labelsep\itshape
    #1\@addpunct{.\enspace\textemdash}]\ignorespaces
}{%
  \qed\endtrivlist
}

\newenvironment{proof*}[1][\proofname]{\par
  \normalfont
  \topsep6\p@\@plus6\p@ \trivlist
  \item[\hskip\labelsep\itshape
    #1\@addpunct{.\enspace\textemdash}]\ignorespaces
}{%
  \endtrivlist
}

\newcommand{\pushright}[1]{\ifmeasuring@#1\else\omit\hfill$\displaystyle#1$\fi\ignorespaces}

\def\@secnumfont{\bfseries}
\makeatother

\newtheoremstyle{bfth}
{}
{}
{\itshape}
{\parindent}
{\itshape}
{\textbf{.}\hspace{.5em}\textemdash}
{.5em}
{{\thmname{#1}\upshape\thmnumber{ \textbf{(#2)}}\thmnote{ (\textit{#3})}}}

\newtheoremstyle{bfthcite}
{}
{}
{\itshape}
{\parindent}
{\itshape}
{\textbf{.}\hspace{.5em}\textemdash}
{.5em}
{{\thmname{#1}\upshape\thmnumber{ \textbf{(#2)}}\thmnote{ \textit{#3}}}}

\newtheoremstyle{bfthstar}
{}
{}
{\itshape}
{\parindent}
{\itshape}
{.\hspace{.5em}\textemdash}
{.5em}
{{\thmname{#1}\upshape\thmnumber{ \textbf{(#2)}}\thmnote{ (\textit{#3})}}}

\newtheoremstyle{bfthstarcite}
{}
{}
{\itshape}
{\parindent}
{\itshape}
{.\hspace{.5em}\textemdash}
{.5em}
{{\thmname{#1}\upshape\thmnumber{ \textbf{(#2)}}\thmnote{ \textit{#3}}}}

\theoremstyle{bfth}

\newtheorem{proposition-definition}[theorem]{Proposition-Definition}

\theoremstyle{bfthcite}
\newtheorem{theorem-cite}{Theorem}[section]
\newtheorem{corollary-cite}[theorem]{Corollary}
\newtheorem{lemma-cite}[theorem]{Lemma}
\newtheorem{proposition-cite}[theorem]{Proposition}
\newtheorem{proposition-definition-cite}[theorem]{Proposition-Definition}

\theoremstyle{bfthstar}
\newtheorem*{theorem*}{Theorem}
\newtheorem*{corollary*}{Corollary}
\newtheorem*{lemma*}{Lemma}
\newtheorem*{proposition*}{Proposition}
\newtheorem*{proposition-definition*}{Proposition-Definition}
\newtheorem*{conjecture*}{Conjecture}

\theoremstyle{bfthstarcite}
\newtheorem*{theorem-cite*}{Theorem}
\newtheorem*{corollary-cite*}{Corollary}
\newtheorem*{lemma-cite*}{Lemma}
\newtheorem*{proposition-cite*}{Proposition}
\newtheorem*{proposition-definition-cite*}{Proposition-Definition}

\theoremstyle{definition}

\newtheorem*{definition*}{Definition}
 
\theoremstyle{remark}

\newtheorem*{remark*}{Remark}

\newtheorem*{question*}{Question}

\newtheorem*{questions*}{Questions}

\DeclareMathOperator{\Aut}{Aut}
\DeclareMathOperator{\Tr}{Tr}
\DeclareMathOperator{\Hom}{Hom}

\DeclareMathOperator{\id}{id}

\DeclareMathOperator{\Spec}{Spec}
\DeclareMathOperator{\Gal}{Gal}

\DeclareMathOperator{\supp}{supp}

\DeclareMathOperator{\fib}{fib}
\DeclareMathOperator{\cofib}{cofib}

\DeclareMathOperator{\colim}{colim}

\DeclareMathOperator{\charac}{char}

\DeclareMathOperator{\Map}{Map}

\DeclareMathOperator{\ev}{ev}
\DeclareMathOperator{\coev}{coev}
\DeclareMathOperator{\Frac}{Frac}
\DeclareMathOperator{\Perv}{Perv}
\DeclareMathOperator{\Cone}{Cone}

\newcommand{\Cat}{\mathrm{Cat}}

\newcommand{\et}[1]{{#1}_{\text{\upshape \'et}}}

\newcommand{\xra}{\xrightarrow}
\newcommand{\lra}{\longrightarrow}

\newcommand{\sHom}{\text{\usefont{U}{BOONDOX-cal}{m}{n}Hom}}

\newcommand{\proet}{\text{\upshape pro\'et}}
\newcommand{\lis}{\mathrm{lisse}}

\newcommand{\op}{\mathrm{op}}

\long\def\comment#1{}

\makeatletter
\newcommand*\dotp{{\mathpalette\dotp@{.5}}}
\newcommand*\dotp@[2]{\mathbin{\vcenter{\hbox{\scalebox{#2}{$\m@th#1\bullet$}}}}}
\newcommand{\leqnomode}{\tagsleft@true}
\newcommand{\reqnomode}{\tagsleft@false}

\DeclareRobustCommand{\cev}[1]{%
  {\mathpalette\do@cev{#1}}%
}
\newcommand{\do@cev}[2]{%
  \vbox{\offinterlineskip
    \sbox\z@{$\m@th#1 x$}%
    \ialign{##\cr
      \hidewidth\reflectbox{$\m@th#1\vec{}\mkern4mu$}\hidewidth\cr
      \noalign{\kern-\ht\z@}
      $\m@th#1#2$\cr
    }%
  }%
}
\makeatother

\title{The singular support of an $\ell$-adic sheaf}
\author{Owen Barrett}
\date{\today}
\begin{document}
\begin{abstract}
	We define the singular support of an $\ell$-adic sheaf on a smooth
	variety over any field.
	To do this, we combine Beilinson's construction of the singular support
	for torsion \'etale sheaves with Hansen and Scholze's theory of
	universal local acyclicity for $\ell$-adic sheaves.
\end{abstract}
\maketitle
{\setlength{\parskip}{0.5em}\setstretch{1.0}\small\tableofcontents\par}
\section{Introduction}
The notion of singular support (or micro-support) has for several decades
been the point of departure (see Kashiwara-Schapira \cite{KS}) for the 
microlocal study of sheaves and $\mathcal D$-modules on manifolds.
It was hoped for many years that an analogue of the singular support 
existed for algebraic varieties in positive characteristic.
In 2015, Beilinson \cite{Sasha} realized this hope, constructing the 
singular support of a constructible \'etale sheaf with torsion coefficients
on a smooth variety over any field. One of his many innovations was to
relate the singular support to a broader notion of (universal) local
acyclicity than what is captured by the vanishing cycles.
At the time of his article, a satisfying formalism of universal local
acyclicity did not yet exist for $\ell$-adic sheaves.
The fundamental nature of the singular support, however, fueled the
ambition that it might be defined for $\ell$-adic coefficients, rational
coefficients in particular.
(As we will see, the case of integral coefficients offers essentially 
nothing new.)
Indeed, since the appearance of Beilinson's article,
Umezaki-Yang-Zhao \cite{UYZ}, extending Saito~\cite{Saito}, defined the
characteristic cycle of a $\overline\QQ_\ell$-sheaf.
Recent innovations in the theory of universal
local acyclicity by Lu-Zheng \cite{LZ_ula} and Hansen-Scholze \cite{HS}
have extended the definition of universal local acyclicity to $\ell$-adic
sheaves on a very broad class of schemes. With this new formalism in hand,
it is now possible to define the singular support of an $\ell$-adic sheaf
by extending Beilinson's original arguments to the $\ell$-adic setting.

I wish to extend my heartfelt thanks to
Sasha Beilinson,
Denis-Charles Cisinski,
David Nadler,
David Hansen,
Arthur Ogus
and Peter Scholze
for clarifying explanations, comments and inspiring discussions.
I especially wish to thank Sasha Beilinson for pointing out to me that
the perverse t-exactness of the vanishing cycles implies that they send
torsion-free perverse sheaves to torsion-free perverse sheaves, which is
the heart of the proof of part~\ref{th:SS:integral_model} of
Theorem~\ref{th:SS}.

\subsection{}\label{sec:notation_categories}
Fix a field $k$ and a prime number $\ell\ne\charac k$. Henceforth, 
`variety' will mean `$k$-scheme of finite type.' Let $\Lambda$ be an
algebraic extension $E$ of $\QQ_\ell$ or the ring of integers 
$\mathcal O_E$ of such an extension. We wish to define the singular support
of a constructible complex of $\Lambda$-sheaves on a
variety $X$. If $\Lambda=\mathcal O_E$ for $E$ a finite extension of
$\QQ_\ell$, the classical approach to constructing the derived category of 
$\Lambda$-constructible sheaves on a variety involved taking an inverse
limit of the corresponding categories with coefficients in 
$\mathcal O_E/\varpi^n$ ($\varpi$ a uniformizer of $\mathcal O_E$) as
$n\to\infty$. To obtain the derived category of $E$-constructible sheaves,
one subsequently localized, inverting $\ell$. This approach had the 
disadvantage that the objects under consideration were no longer sheaves on
some site, but were rather formal inverse limits of such. This artificiality 
caused problems when one wanted to replace the variety $X$ by a space which
was no longer of finite type over a field, such as the Milnor tubes which
appear in the definition of local acyclicity.

To address these and other issues, Bhatt and Scholze \cite{BS} in 2015
defined a Grothendieck topology on schemes -- finer than the \'etale
topology -- called the pro-\'etale topology, and showed that the classical
derived categories of constructible $\ell$-adic sheaves could be recovered
as full subcategories of $D(X_\proet,\Lambda)$, where this notation
literally means the derived category of sheaves of modules over a certain
sheaf of rings $\Lambda=\Lambda_X$. This sheaf of rings on $X_\proet$ is
obtained by pulling back the condensed ring corresponding to the topological
ring $\Lambda$ along $X_\proet\to\ast_\proet$; concretely, if $U\to X$ is
weakly \'etale (i.e. for $U$ in $X_\proet$), 
$\Gamma(U,\Lambda_U)=\Map_{\mathrm{cont}}(U,\Lambda)$.
Bhatt and Scholze's definition of constructible sheaf was later extended by
Hemo, Richarz and Scholbach \cite{HRS}, and it is their definition that we
now present.

Fix a scheme $X$. By abuse of notation we will continue to use $\Lambda$ 
to denote the sheaf $\Lambda_X$ defined above
(despite the notation, this sheaf is generally not constant; we will refer
to the constant sheaf with value $\Lambda$ as $\underline\Lambda$). We let
$\mathcal D(X_\proet,\Lambda)$ denote the derived category of 
$\Lambda$-sheaves on $X_\proet$, considered as a stable 
$\infty$-category.\footnote{So, for example, the cone of some map in
$D(X_\proet,\Lambda)$ computes the (homotopy) cofiber or cokernel.}
Its homotopy category is triangulated and will be denoted by 
$D(X_\proet,\Lambda)$. The category
$\mathcal D(X_\proet,\Lambda)$ carries a symmetric monoidal structure given 
by tensor product, and we say an object in $\mathcal D(X_\proet,\Lambda)$
is a \emph{lisse sheaf} if it's dualizable with respect to this structure.
(Concretely, $\mathcal F$ in $\mathcal D(X_\proet,\Lambda)$ is dualizable
if and only if the natural map
$\mathcal G\Lotimes_\Lambda R\sHom(\mathcal F,\Lambda)\to
R\sHom(\mathcal F,\mathcal G)$
is an isomorphism for every $\mathcal G\in\mathcal D(X_\proet,\Lambda)$.)
We let $\mathcal D_\lis(X,\Lambda)=\mathcal D_\lis(X)$ denote the full
subcategory of $\mathcal D(X_\proet,\Lambda)$ on the lisse sheaves. We say
that an object $\mathcal F$ of $\mathcal D(X_\proet,\Lambda)$ is a
\emph{constructible sheaf} if every affine open $U\subset X$ admits a
finite partition $U=\coprod_i U_i$ into constructible locally closed
subschemes $U_i$ (called strata) so that $\mathcal F|_{U_i}$ is lisse.
We let $\mathcal D(X):=\mathcal D(X,\Lambda):=\mathcal D_{\mathrm{cons}}(X_\proet,\Lambda)$ denote
the full subcategory of $\mathcal D(X_\proet,\Lambda)$ on the constructible 
sheaves and we refer to objects of $\mathcal D(X)$ simply as 
\emph{sheaves}.

The standard ($p=0$) t-structure on $\mathcal D(X_\proet,\Lambda)$ often
restricts to a t-structure on $\mathcal D(X)$, and we will refer to sheaves in the heart as
\emph{abelian} sheaves to emphasize that they're concentrated in degree
zero and form an abelian category. With appropriate finiteness assumptions,
the six functors can be defined on $\mathcal D(X)$, and we won't decorate
these functors with Rs or Ls, so $j_*:=Rj_*$, $\sHom:=R\sHom$,
$\otimes:=\Lotimes$ etc. No confusion can result as we will explicitly flag 
any appearance of the un-derived versions of these functors.

\subsection{}\label{sec:intro_ula}
Let $f:X\to S$ be a map locally of finite presentation between schemes over
$\Spec\ZZ[\frac1\ell]$ and let
$\mathcal F\in D(X)$.
\begin{definition*}
We say $f$ is
\emph{universally locally acyclic relative to $\mathcal F$}
(or that $\mathcal F$ is universally locally acyclic over $S$ if $f$ is
understood) if for every geometric point $x\to X$ and generization $t\to S$
of $f(x)$, the map
$\mathcal F_x:=\Gamma(X_x,\mathcal F)\to\Gamma(X_x\times_{S_{f(x)}}t,\mathcal F)$
is an isomorphism, and this remains true after arbitrary base change in
$S$.
\end{definition*}
That this definition, formally identical to the classical one, holds any
water is a consequence of the work of Hansen and Scholze, and will be 
discussed in \S\ref{sec:ula}.

When $S$ is a nonsingular curve (or more generally a regular scheme of
dimension 1), we can define the notion of universal local acyclicity at a
point in terms of vanishing cycles. Let $x\to X$ be a geometric point and
denote by $X_x$ and $S_{f(x)}$ the strict localizations of $X$ and $S$ at
$x$ and $f(x)$, respectively. As $S_{f(x)}=\Spec A$ with $A$ a strictly
henselian dvr, the normalization $V$ of $A$ in $\alg{(\Frac A)}$ is an
absolutely integrally closed valuation ring\footnote{i.e. $V$ is a
valuation ring and $\Frac V$ is algebraically closed.}
(with value group $\QQ$). Denoting by $j$ and $i$ the open and
closed immersions of the generic point $\eta$ and closed point of $V$, 
respectively, and base extensions of such, the sheaf of nearby cycles
$\psi_f(\mathcal F)$ on the geometric special fiber $f^{-1}(f(x))$ is
defined as $i^*j_*(\mathcal F|_{X_\eta})$.\footnote{The nearby cycles are
usually defined using the normalization $V'$ of $A$ in $\sep{(\Frac A)}$
in lieu of $V$. The distinction is immaterial as $V'\to V$ is radicial.
Indeed, let $K:=\Frac V$, $K':=\Frac V'$, let $v$ denote the valuation on
$K'$, and suppose $\charac k=:p>0$.
Then every $a\in K$ satisfies $a^{p^j}\in K'$ for some $j$, so the
valuation $v$ on $K'$ extends uniquely to one on $K$ (still called $v$), so
that $v(a)\geq0\Leftrightarrow a^{p^j}\in V'$. If $\mathfrak p\subset V'$
is a prime ideal, $V/\mathfrak p V$ is generated over $V'/\mathfrak p$ by
elements $\overline a$ satisfying equations like 
$\overline a^{p^j}=\overline b\in V'/\mathfrak p$. As $V$ and $V'$ are
valuation rings, their spectra are linearly ordered, and as the extension
$V'\subset V$ is integral, there are no nontrivial specializations between
points in any fiber, whence $|\Spec V\otimes_{V'}k(\mathfrak p)|=1$. As
radical ideals in valuation rings are prime, the point of the 
$\mathfrak p$-fiber is $\sqrt{\mathfrak p V}$, and
$k\big(\sqrt{\mathfrak p V}\big)/k(\mathfrak p)$ is purely inseparable as
the extension $V/\sqrt{\mathfrak p V}\supset V'/\mathfrak p$ is generated
by $p^j$-th roots for various $j$.}
It's a constructible sheaf
\cite[Theorem 4.1]{HS} endowed with a map from $i^*\mathcal F$, and the
sheaf of vanishing cycles $\phi_f(\mathcal F)\in D(f^{-1}f(x))$ is
obtained as $\cofib(i^*\mathcal F\to\psi_f(\mathcal F))$.
We say $f$ is \emph{universally locally acyclic at $x$ relative to 
$\mathcal F$} if $\phi_f(\mathcal F)_x=0$.
This is tantamount to asking that the map
$\mathcal F_x=\Gamma(X_x,\mathcal F)\to\Gamma(X_x\times_{S_{f(x)}}\eta,\mathcal F)$
be an isomorphism.

By abuse of notation, when speaking of varieties we'll sometimes refer to
universal local acyclicity at a scheme-theoretic point rather than at a
geometric point. There is no ambiguity: if $x\in X$ is a scheme-theoretic
point, then $\Gal(\sep k/k)$ acts transitively on (algebraic) geometric
points centered on $x$, and induces isomorphisms on the corresponding stalks
of $\phi_f(\mathcal F)$. Given $x$, we'll often use $\overline x$ to denote
a geometric point centered on $x$ defined by a separable closure of $k(x)$.

\subsection{}\label{sec:transversality_def}
In this section and the next one we recall some terminology from
\cite{Sasha}.

If $X$ is a variety, we refer to correspondences of
the form $X\xleftarrow h U\xrightarrow fY$ as \emph{test pairs} on $X$.
If $\mathcal F\in D(X)$, we say a test pair $(h,f)$ is 
\emph{$\mathcal F$-acyclic} if $f$ is universally locally acyclic 
rel.\,$h^*\mathcal F$.\footnote{Beilinson's original definition of
$\mathcal F$-acyclicity for \'etale sheaves was phrased in terms of local
acyclicity rather than universal local acyclicity. Gabber later proved they
are the same notion for a morphism of finite type with noetherian target 
\cite[Corollary 6.6]{LZ_bases}.}

Suppose $X$ is a smooth variety and $C\subset T^*X:=T^*(X/k)$ a closed
conical (i.e. $\GG_m$-invariant) subset of the cotangent bundle of $X$.
The \emph{base} of $C$ is its image in $X$, which is closed.
A morphism $h:U\to X$ with $U$ smooth is said to be $C$-transversal at a
geometric point $u\to U$ if for every nonzero $\mu\in C_{h(u)}$ the
covector $dh(\mu)\in T^*_uU$ is nonzero. A map $f:X\to Y$ with $Y$ smooth
is said to be $C$-transversal at a geometric point $x\to X$ if for every
nonzero $\nu\in T^*_{f(x)}Y$ one has $df\notin C_x$. The morphism $f$ or $h$
is said to be $C$-transversal if it is so at every geometric point.

A test pair $(h,f)$ as above is said to be $C$-transversal at a
geometric point $u\to U$ if $U$ and $Y$ are smooth, $h$ is $C$-transversal
at $u$ and $f$ is $h^\circ C_u$-transversal at $u$, where 
$h^\circ C_u:=dh(C_{h(u)})$. The test pair is said to be $C$-transversal
if it is so at every geometric point of $U$.
If $h:U\to X$ is $C$-transversal (as is the case when $h$ is smooth), then
for each geometric point $u\to U$, the closed subsets 
$h^\circ C_u\subset T^*_uU$ are the fibers of a closed
conical subset of $T^*U$, denoted $h^\circ C$ \cite[Lemma 1.2(ii)]{Sasha}.

\subsection{}\label{sec:ss_def}
Let $X$ be a smooth variety and $\mathcal F\in D(X)$. We say that 
$\mathcal F$ is \emph{micro-supported} on a closed conical subset 
$C\subset T^*X$ if every $C$-transversal test pair is $\mathcal F$-acyclic.
If there is a smallest closed conical subset $C\subset T^*X$ on which 
$\mathcal F$ is micro-supported, we say 
\emph{$\mathcal F$ has singular support} and call $C$ the
\emph{singular support} of $\mathcal F$ and denote it by $SS(\mathcal F)$.

We say that $\mathcal F$ is \emph{weakly micro-supported} on a closed 
conical subset $C\subset T^*X$ if every $C$-transversal test pair $(h,f)$
satisfying the following additional conditions is $\mathcal F$-acyclic:
\begin{itemize}[label={--}]
\item The morphism $f:U\to Y$ has target $Y=\A^1$; and
\item If $k$ is infinite then $h:U\to X$ is an open immersion; if $k$ is
finite then $h$ factors as a composition 
$U=V_{k'}:=V\otimes_kk'\to V\hookrightarrow X$ of a base change of a finite 
extension $k'\supset k$ followed by an open immersion.
\end{itemize}
It's not hard to see that a sheaf $\mathcal F$ is both micro-supported and
weakly micro-supported on $T^*X$; we will review why later.
It's also not hard to see that the set of all closed conical subsets of
$T^*X$ on which $\mathcal F$ is weakly micro-supported has a smallest
element, denoted $SS^w(\mathcal F)$. If $\mathcal F$ has singular support, 
then it's clear that $SS^w(\mathcal F)\subset SS(\mathcal F)$. It follows
from Lemma~\ref{lem:ula_curve} that $SS^w(\mathcal F)$ admits the following
description: when $k$ is infinite, 
$SS^w(\mathcal F)$ is the closure in $T^*X$ of the set of points of the form
$(x,df(x))$, where $x\in|X|$ and $f$ is a function on a Zariski neighborhood
of $x$ which is not universally locally acyclic at $x$ 
rel.\,$\mathcal F$.\footnote{i.e. for any geometric point $\overline x$
centered on $x$, the vanishing cycles $\psi_f(\mathcal F)$ has nonzero stalk
at $\overline x$. That it suffices to check on closed points follows from
the constructibility of the sheaf of vanishing cycles.}
When $k$ is finite, $SS^w(\mathcal F)$ coincides with the image under the
map $T^*X_{\alg k}\to T^*X$ of the closure of the set of points of the form
$(x,df(x))$ where $x\in|X_{\alg k}|$ and $f$ is a function on $V_{\alg k}$,
$V\subset X$ a Zariski neighborhood of $x$, such that $f$ is not universally
locally acyclic at $x$ rel.\,$\mathcal F$.\footnote{This closed conical
subset coincides with the one built from finite extensions. Indeed, let
$k'$ be a finite extension of $k$ and let $J_{k'}$ denote the image in 
$T^*X$ of the set $\{(x,df(x))\mid x\in|X_{k'}|,$ $f$ a function on 
$V_{k'}$, $V\subset X$ a Zariski neighborhood of (the image of) $x$, such
that $f$ is not universally locally acyclic at $x$
rel.\,$\mathcal F\}\subset T^*X_{k'}$.
Then $SS^w(\mathcal F)=\overline{\cup_{k'}J_{k'}}$.

To see that $SS^w(\mathcal F)$ defined in this way is conical, first note 
that $\cup_{k'}J_{k'}$ is conical: given a geometric closed point 
$\mu\to\cup_{k'}J_{k'}$ and an $\lambda\in\GG_m(\alg k)$, there is a
(possibly different) geometric point $\nu\to\cup_{k'}J_{k'}$ centered on the 
same closed point of $T^*X$ as $\mu$ such that $\nu=(x,df(x))$ for $f$
defined on some neighborhood $V_{\alg k}$ of $x\in X_{\alg k}$.
As $\Gal(\alg k/k)$ acts on $T^*X_{\alg k}$, fixing $T^*X$, there is some 
$\sigma\in\Gal(\alg k/k)$ so that $\mu=\sigma(\nu)$. 
As $\sigma^{-1}(\lambda)f$ is not universally locally acyclic at $x$ 
rel.\,$\mathcal F$, the image of $\sigma^{-1}(\lambda)\nu$ in $T^*X$ lies in
$SS^w(\mathcal F)$, so $\lambda\sigma(\nu)=\lambda\mu$ does, too.
To conclude, observe that the closure of a $\GG_m$-stable subset of 
$T^*X_{\alg k}$ is again $\GG_m$-stable.}

\subsection{}\label{sec:main_results}
We can now state our main results. Fix a smooth variety $X$.
If $\mathcal F\in D(X)$ with $\Lambda$ integral (i.e. so that 
$\Lambda\otimes\ZZ/\ell\ne0$), then
$\mathcal F=\mathcal F'\otimes_{\mathcal O_E}\Lambda$ for some
$\mathcal F'\in D(X,\mathcal O_E)$, where $E$ is a finite extension of
$\QQ_\ell$. Then $\mathcal F'/\ell:=\mathcal F'\otimes_{\ZZ_\ell}\ZZ/\ell$
is an \'etale sheaf of $\ZZ/\ell$-modules, and therefore has singular
support as defined by Beilinson.
If $\Lambda$ is rational (i.e. $\Lambda$ is an algebraic extension $E$ of 
$\QQ_\ell$), then $\mathcal F$ admits an integral model $\mathcal F_0$ by
\cite[Corollary 2.4]{HS}; i.e. $\mathcal F_0$ is in $D(X,\mathcal O_E)$ and
$\mathcal F_0\otimes_{\ZZ_\ell}\QQ_\ell\simeq\mathcal F$.
\begin{theorem*}\label{th:SS}\begin{enumerate}[label=(\roman*)]
	\item\label{th:SS:existence} Every sheaf $\mathcal F$ in $D(X)$ has
	singular support.
	\item If $\Lambda=\mathcal O_E$ with $E/\QQ_\ell$ finite, then 
	$SS(\mathcal F)=SS(\mathcal F/\ell)$.
	\item If $\Lambda=\mathcal O_L$ with $L/\QQ_\ell$ infinite algebraic,
	then $SS(\mathcal F)=SS(\mathcal F')$, where $\mathcal F'$ is any
	approximation for $\mathcal F$ over the ring of integers of a finite
	extension of $\QQ_\ell$.
	\item\label{th:SS:integral_comparison} If $\Lambda$ is rational and
	$\mathcal F_0$ is any integral model for $\mathcal F$, then
	$SS(\mathcal F)\subset SS(\mathcal F_0)$.
	\item\label{th:SS:integral_model} If $\Lambda$ is rational then 
	$\mathcal F$ admits an integral model $\mathcal F_0$ with
	$SS(\mathcal F_0)=SS(\mathcal F)$.
	If $\mathcal F$ is perverse, then any integral model $\mathcal F_0$ for
	$\mathcal F$ which is perverse and torsion-free as a perverse
	sheaf\footnote{i.e. such that $\ell:\mathcal F_0\to\mathcal F_0$ is
	injective in $\Perv(X)$.} has $SS(\mathcal F_0)=SS(\mathcal F)$.
	\item\label{th:SS:dim} If $X$ is connected and $\mathcal F\ne0$, then
	$SS(\mathcal F)$ is equidimensional of dimension $\dim X$.
	\item\label{th:SS:SSw} One has $SS^w(\mathcal F)=SS(\mathcal F)$.
	\item\label{th:SS:constituents} One has
	$SS(\mathcal F)=\bigcup_iSS({^p\mathcal H^i}\mathcal F)$.
	If $\Lambda$ is rational and $\{\mathcal F_\alpha\}$ are the perverse
	Jordan-H\"older constituents of $\mathcal F$,\footnote{An irreducible
	perverse sheaf is a \emph{constituent} for $\mathcal F$ if it appears in
	a Jordan-H\"older series for some ${^p\mathcal H^i\mathcal F}$.}
	then moreover $SS(\mathcal F)=\bigcup_\alpha SS(\mathcal F_\alpha)$.
	\item\label{th:SS:verdier} $SS(D\mathcal F)=SS(\mathcal F)$, where $D$
	denotes Verdier duality.
	\item\label{th:SS:smooth_pullback} For a smooth map $f:Y\to X$ one has
	$SS(f^*\mathcal F)=f^\circ SS(\mathcal F)$.
	\item\label{th:SS:field_extn} If $k'/k$ is any extension of fields and 
	$\mathcal F_{k'}$ is the inverse image of $\mathcal F$ on
	$X_{k'}:=X\otimes_kk'$, then the closed subsets $SS(\mathcal F)_{k'}$
	and $SS(\mathcal F_{k'})$ of $T^*X_{k'}$ coincide.
	\end{enumerate}
\end{theorem*}
It continues to be true that if $X$ is a projective space, the irreducible
components of the singular support are in bijection with those of the
ramification divisor of the Radon transform of $i_*\mathcal F$, where $i$
is a Veronese embedding of degree $>1$. The precise statement will be given
in \S\ref{sec:SS}.

The theorem shows that the $\ell$-adic singular support enjoys the
same formal properties as the singular support of torsion \'etale sheaves.
In particular, when $\mathcal F$ is a rational sheaf, the inclusion
$SS(\mathcal F)\subset SS(\mathcal F_0)$
combined with the equidimensionality of both implies that $SS(\mathcal F)$
can be obtained from $SS(\mathcal F_0)$ by deleting some irreducible
components; the equality $SS(\mathcal F)=SS^w(\mathcal F)$ shows that these
deleted irreducible components correspond to covectors $(x,df(x))$ where
$\phi_f(\mathcal F)_{\overline x}$ is nonzero but of torsion; i.e. where 
$\phi_f(\mathcal F)_{\overline x}\ne0$ but 
$\phi_f(\mathcal F)_{\overline x}\otimes\QQ_\ell=0$, or, equivalently, as
$\phi_f(\mathcal F)$ is constructible, so that 
$\phi_f(\mathcal F)_{\overline x} $ is a $\ZZ/\ell^n$-module for some $n$.
It's possible to describe these deleted components more explicitly; see the 
proposition below.

If $\mathcal F$ is a rational sheaf and $\mathcal F_0$ is any integral model
for it, and $\mathcal G_0$ is a torsion \'etale constructible sheaf so that
$SS(\mathcal G_0)\not\subset SS(\mathcal F)$, then 
$\mathcal F_0\oplus\mathcal G_0$ is an example of an integral model for
$\mathcal F$ such that the inclusion 
$SS(\mathcal F)\subset SS(\mathcal F_0\oplus\mathcal G_0)$ is strict.

\begin{proof}[Proof of Theorem~\ref{th:SS}\ref{th:SS:integral_model}]
	As any approximation for $\mathcal F$ will have the same singular 
	support (cf. Lemma~\ref{lem:integral_extension_ula}), we may replace 
	$\mathcal F$ by such an approximation and assume $\Lambda$ is a finite
	extension $E$ of $\QQ_\ell$. Suppose first $\mathcal F$ is a perverse
	sheaf and let $\mathcal F'_0$ be any perverse integral model.
	Then $\ker(\ell^n:\mathcal F'_0\to\mathcal F'_0)$
	eventually stabilizes with stable value $\mathcal K_0$ as 
	$\Perv(X,\mathcal O_E)$ is noetherian.
	Then $\mathcal F_0:=\mathcal F'_0/\mathcal K_0$ is an integral model for
	$\mathcal F$ which is torsion-free as a perverse sheaf.
	I claim $SS^w(\mathcal F_0)=SS^w(\mathcal F)$.
	First note that the bases of both are the same (this follows from
	Lemma~\ref{lem:B2.1}\ref{lem:B2.1:base} and the next paragraph).
	Next note that we may assume $k$ is infinite by
	part~\ref{th:SS:field_extn} of the theorem.
	Now suppose $SS^w(\mathcal F_0)$ were strictly larger than 
	$SS^w(\mathcal F)$, so that we could find a geometric point $x\to X$ and
	a smooth function $f:U\to\A^1$ on a Zariski neighborhood $U\subset X$ of
	$x$ such that $\phi_f(\mathcal F_0)_x\ne0$ but 
	$(x,df(x))\notin SS^w(\mathcal F)$.
	By the definition of $SS^w(\mathcal F)$, this would imply that we could
	shrink $U$ about $x$ so that $\phi_f(\mathcal F)=0$ on 
	$U\cap f^{-1}f(x)$.
	The perverse (left) t-exactness of $\phi_f[-1]$ implies that 
	$\phi_f(\mathcal F_0)[-1]$ is torsion-free as a perverse sheaf.
	This is a contradiction as a nonzero torsion-free perverse sheaf
	has a nonzero rationalization.
	
	Indeed, let $W$ be a variety and $\mathcal M_0$ a torsion-free perverse 
	$\mathcal O_E$-sheaf on it. If $V\subset W$ is a Zariski open such that
	$\mathcal M_0|_V$ is nonzero and lisse, then $\mathcal M_0|_V$ is a
	shift of a torsion-free local system (i.e. abelian lisse sheaf) on $V$
	(see Corollary~\ref{cor:lisse_t}). If $\mathcal M_0$ is not supported
	generically, then let $i:Z\hookrightarrow W$ denote its support, so
	that $\mathcal M_0=i_*\mathcal N_0$ for some torsion-free perverse
	$\mathcal O_E$-sheaf on $Z$. Then $\mathcal N_0$ is supported 
	generically on $Z$ and we may conclude by the previous argument.
	
	For a general $\mathcal F\in D(X,E)$, we may write $\mathcal F$
	as a successive extension in $D(X,E)$ of (shifts of) its perverse
	cohomology sheaves $^p\mathcal H^i\mathcal F$. For each $i$, we can find
	a perverse integral model $\mathcal F_0^i$ for 
	$^p\mathcal H^i\mathcal F$ with
	$SS(\mathcal F_0^i)=SS({^p\mathcal H^i\mathcal F})$, as these coincide
	with their weak versions by part~\ref{th:SS:SSw} of the theorem.
	Proceeding by induction on the perverse amplitude of $\mathcal F$,
	we may assume we've found integral models $\mathcal G_0$ and
	$\mathcal G_0'$ for $^p\tau^{\leq 0}\mathcal F$ and
	$^p\tau^{>0}\mathcal F$ with 
	$SS(\mathcal G_0)=SS({^p\tau^{\leq 0}}\mathcal F)$ and
	$SS(\mathcal G_0')=SS({^p\tau^{>0}}\mathcal F)$.
	Multiplying by a suitable power of $\ell$, we may assume that the map
	${^p\tau^{>0}}\mathcal F[-1]\to{^p\tau^{\leq0}}\mathcal F$ is the
	rationalization of some map $\varrho:\mathcal G_0'[-1]\to\mathcal G_0$
	in $D(X,\mathcal O_E)$ (see \S\ref{sec:cons}), so
	$\mathcal F_0:=\Cone(\varrho)$ has 
	$\mathcal F_0[\ell^{-1}]\simeq \mathcal F$.
	Part~\ref{th:SS:constituents} of the theorem implies that 
	$SS(\mathcal F)
	=SS(^p\tau^{\leq 0}\mathcal F)\cup SS(^p\tau^{>0}\mathcal F)
	=SS(\mathcal G_0')\cup SS(\mathcal G_0')$.
	On the other hand, the sheaves in $D(X,\mathcal O_E)$
	micro-supported on some fixed conical closed subset form a thick
	subcategory of $D(X,\mathcal O_E)$ by 
	Lemma~\ref{lem:B2.1}\ref{lem:B2.1:thick}, in particular closed under 
	cones and shifts. Therefore
	$SS(\mathcal F_0)\subset SS(\mathcal G_0')\cup SS(\mathcal G_0')
	=SS(\mathcal F)$.
	The reverse inclusion by part~\ref{th:SS:integral_comparison} of the
	theorem.
\end{proof}
If $\mathcal F_0$ is an sheaf with $\mathcal O_E$-coefficients on a
smooth variety $X$, the above argument makes it possible to describe
explicitly the discrepancy between $SS(\mathcal F_0)$ and 
$SS(\mathcal F)$. We have by part~\ref{th:SS:constituents} of the
theorem that
$SS(\mathcal F_0)=\bigcup_iSS({^p\mathcal H^i}\mathcal F_0)$
and likewise for $\mathcal F:=\mathcal F_0[\ell^{-1}]$
(rationalization is perverse t-exact).
Form as above for each $i$ the maximal torsion perverse subsheaf
$\mathcal K_0^i\subset{^p\mathcal H^i}\mathcal F_0$
(i.e. so that $\mathcal K_0^i=\ker(\ell^m:{^p\mathcal H^i}\mathcal F_0
\to{^p\mathcal H^i}\mathcal F_0)$ for $m\gg0$),
and let $\mathcal G_0^i:=({^p\mathcal H^i}\mathcal F_0)/\mathcal K^i_0$.
Then $\mathcal G_0^i$ is a torsion-free perverse sheaf and an integral
model for $^p\mathcal H^i\mathcal F$, so by
part~\ref{th:SS:integral_model} of the theorem,
$SS(\mathcal G_0^i)=SS({^p\mathcal H^i}\mathcal F)$.
The argument in the proof above shows that if
$0\to\mathcal M_0\to\mathcal N_0\to\mathcal P_0\to0$ is an exact 
sequence of perverse sheaves on $X$, then 
$SS(\mathcal N_0)=SS(\mathcal M_0)\cup SS(\mathcal P_0)$.
Therefore $SS({^p\mathcal H^i}\mathcal F)$ is obtained from
$SS({^p\mathcal H^i}\mathcal F_0)$ by deleting the irreducible
components of $SS(\mathcal K_0^i)$ which are not already in 
$SS({^p\mathcal H^i}\mathcal F)$.
This proves the following
\begin{proposition*}
	Let $X$ be a smooth variety and $\mathcal F_0$ be in
	$D(X,\mathcal O_E)$ with $E$ an algebraic extension of $\QQ_\ell$.
	For each $i\in\ZZ$ let $\mathcal K_0^i$ be the maximal torsion perverse
	subsheaf of $^p\mathcal H^i\mathcal F_0$. Then 
	$SS(\mathcal F_0)=
	SS(\mathcal F_0\otimes_{\mathcal O_E}E)\cup\bigcup_i SS(\mathcal K_0^i)$.
\end{proposition*}

\subsection{} Finally, we record a useful result concerning
$\mathcal F$-transversality.
Let $h:W\to X$ be a locally finitely-presented map of locally
quasi-excellent $\ell$-coprime\footnote{`Locally quasi-excellent' meaning
admitting a Zariski cover by spectra of quasi-excellent rings, 
while `$\ell$-coprime' meaning defined over $\Spec\ZZ[\frac1\ell]$.}
schemes, suppose $X$ is moreover quasi-compact (hence noetherian), and let 
$\mathcal F\in D(X)$. Following Saito \cite[\S8]{Saito}, we say that
$h$ is \emph{$\mathcal F$-transversal} if for every quasi-compact open
$U\subset W$ separated over $X$, the map
$(h|_U)^*\mathcal F\otimes (h|_U)^!\Lambda\to(h|_U)^!\mathcal F$ is an 
isomorphism on $U$, where here $(h|_U)^!$ is defined by virtue of
\cite[Lemma 6.7.19]{BS} and the morphism is the one obtained via adjunction
and the projection formula. This property can be checked over a Zariski
cover of $W$ by quasi-compact opens separated over $X$, and if $W$ is
quasi-compact and $h$ separated, then $h$ is $\mathcal F$-transversal if
and only if the map $h^*\mathcal F\otimes h^!\Lambda\to h^!\mathcal F$ is
an isomorphism on $W$.
Then we have the following result relating $SS(\mathcal F)$-transversality
to $\mathcal F$-transversality.
\begin{proposition-cite*}[{\cite[Proposition 8.13]{Saito}}]
	Let $W$ and $X$ be smooth varieties and $\mathcal F\in D(X)$.
	Then any $SS(\mathcal F)$-transversal morphism $h:W\to X$
	is $\mathcal F$-transversal.
\end{proposition-cite*}
The proof is identical to the one for \'etale sheaves using
Proposition~\ref{prop:illusie} in lieu of \cite[Proposition 2.10]{Illusie}.

\section{Constructible sheaves}\label{sec:cons}
In this section we recall some results concerning the category
$\mathcal D(X)$ from \cite{BS}, \cite{HRS} and \cite{HS}. In this section
$X$ denotes any scheme, unless indicated otherwise.

\subsection{}\label{sec:cons_facts}
Suppose $E$ is a finite extension of $\QQ_\ell$ with ring of integers 
$(\mathcal O_E,\varpi)$. Then 
$\mathcal D(X,\mathcal O_E)=\lim_n\mathcal D(X,\mathcal O_E/\ell^n)
=\lim_n\mathcal D(X,\mathcal O_E/\varpi^n)$
\cite[Proposition 5.1]{HRS},\footnote{As explained in
\cite[\S2.2]{HRS}, the inclusion of the $\infty$-category of
idempotent complete stable $\infty$-categories into that of stable
$\infty$-categories into that of $\infty$-categories preserves small limits 
and filtered colimits, so this limit can be formed in any of these 
$\infty$-categories, and we will say no more about this.}
and if $E$ is now any algebraic extension of $\QQ_\ell$ and $X$ is qcqs,
$\mathcal D(X,\mathcal O_E)=\colim\mathcal D(X,\mathcal O_L)$ and 
$\mathcal D(X,E)=\colim\mathcal D(X,L)$, where the colimit is over those 
finite extensions $E\supset L\supset\QQ_\ell$ \cite[Proposition 5.2]{HRS}.
The same is true with $\mathcal D(X)$ replaced by $\mathcal D_\lis(X)$.

When $E/\QQ_\ell$ is finite, $\mathcal D(X,\mathcal O_E)$ admits another 
description than the one given by \S\ref{sec:notation_categories}: it
coincides with the full subcategory of $\mathcal D(X_\proet,\mathcal O_E)$ 
on the derived $\ell$-complete objects $A$
(i.e. $A=\lim_n(A/\ell^n:=A\otimes_{\ZZ_\ell}\ZZ/\ell^n)$) such that
$A\otimes_\ZZ\ZZ/\ell$ is a perfect-constructible \'etale
sheaf\footnote{i.e. is in the essential image of the functor 
$\mathcal D_\mathrm{cons}(\et X,\mathcal O_E/\ell)\hookrightarrow\mathcal D(\et X,\mathcal O_E/\ell)\to\mathcal D(X_\proet,\mathcal O_E/\ell)$,
where a complex of \'etale $\mathcal O_E/\ell$-sheaves $\mathcal G$ belongs
to $\mathcal D_{\mathrm{cons}}(\et X,\mathcal O_E/\ell)$ if there exists a
finite stratification of $X$ into constructible locally-closed subsets so
that the restriction of $\mathcal G$ to each stratum is \'etale-locally
constant with perfect values.}
\cite[Proposition 7.6]{HRS}.

On any qcqs $X$ on which $\ell$ is invertible and $E/\QQ_\ell$
an algebraic extension, 
$\mathcal D(X,\mathcal O_E)\otimes_{\ZZ_\ell}\QQ_\ell\to\mathcal D(X,E)$ is
an equivalence \cite[Corollary 2.4]{HS}.
The left side coincides with the idempotent completion of the localization
$\mathcal D(X,\mathcal O_E)[\ell^{-1}]$, but for the spaces that appear in
this paper, idempotent completion is superfluous.\footnote{Under weak
finiteness hypotheses to ensure the standard t-structure on
$\mathcal D(X_\proet,\mathcal O_E)$ descends to one on
$\mathcal D(X,\mathcal O_E)$, the Verdier quotient is already
idempotent-complete as the t-structure is bounded \cite[Appendix B.2]{CD}
\cite[Lemma 1.2.4.6]{HA}. See Proposition~\ref{prop:t-structure}.}
If $\mathcal F$ belongs to $\mathcal D(X,E)$, 
we will refer to sheaves $\mathcal F_0$ in $\mathcal D(X,\mathcal O_E)$
such that $\mathcal F_0[\ell^{-1}]=\mathcal F$ as \emph{integral models}
for $\mathcal F$.

Hansen and Scholze \cite[Theorem 2.2]{HS} prove that the
assignment $?\mapsto\mathcal D(?,\Lambda)$
determines an arc-sheaf of $\infty$-categories
$\mathrm{Sch}_{\mathrm{qcqs}}^\op\to\Cat_\infty$.\footnote{Bhatt and Mathew 
\cite{BM} prove the corresponding result for constructible \'etale sheaves.
For the definitions of the arc-topology and the closely related
(and coarser) $v$-topology, see \emph{op. cit.}}
This means that if $Y\to X$ is an arc-cover of qcqs schemes, the map
\begin{equation*}\begin{tikzcd}[column sep=15]
\mathcal D(X)\to\lim_\Delta\Big(\mathcal D(Y^{\bullet/X})
:=\Big(\mathcal D(Y)\arrow[r,shift left]\arrow[r,shift right]
&\mathcal D(Y\times_XY)\arrow[r,shift left=2]\arrow[r]\arrow[r,shift right=2]&\ldots\Big)\Big)
\end{tikzcd}
\end{equation*}
is an equivalence, where this limit is indexed by the simplex category
$\Delta$ and computed in $\Cat_\infty$.
This fact will be used in \S\ref{sec:ula}.

The following standard result allows us to detect that a map of (integral)
constructible sheaves is an isomorphism from the vanishing of its cone 
modulo $\ell$.
\begin{lemma*}\label{lem:reduction_conservative}
	Let $X$ be any scheme and $E$ a finite extension of $\QQ_\ell$.
	Reduction modulo $\ell$ is a conservative functor from the full
	subcategory $\mathcal D_{\mathrm{comp}}(X_\proet,\mathcal O_E)$ of 
	derived $\ell$-complete objects of $\mathcal D(X_\proet,\mathcal O_E)$
	to $\mathcal D(X_\proet,\mathcal O_E/\ell)$.
\end{lemma*}
\begin{proof}
	As $\mathcal D_{\mathrm{comp}}(X_\proet,\mathcal O_E)$ is a stable
	subcategory of $\mathcal D(X_\proet,\mathcal O_E)$, it suffices to show
	that if $A\in\mathcal D_{\mathrm{comp}}(X_\proet,\mathcal O_E)$ has
	$A/\ell=0$, then $A=0$. 
	The exact sequence $0\to\ZZ/\ell^{n-1}\xra\ell\ZZ/\ell^n\to\ZZ/\ell\to0$
	for $n>1$ and induction on $n$ shows that $A/\ell^n=0$ for all $n$. As
	$A=\lim_nA/\ell^n$, we're done.
\end{proof}
\subsection{}\label{sec:stalks}
Let $X$ be any scheme and $x\in X$. If $\mathcal F$ is in 
$\mathcal D(X_\proet,\Lambda)$, we denote by $\mathcal F_{\overline x}$ the
complex $\Gamma(X_{\overline x},\mathcal F)$ in $\mathcal D(\Lambda)$, the
$\infty$-category of chain complexes of $\Lambda$-modules up to
quasi-isomorphism.
Although the site $X_\proet$ has enough points as it gives rise to a 
coherent topos, the stalk functors
$\mathcal F\mapsto\mathcal F_{\overline x}$,
as $x$ runs over the points of $X$, do not form a conservative family of
functors on $\mathcal D(X_\proet,\ZZ)$ \cite[Example 4.2.11]{BS}.
\begin{lemma*}\label{lem:stalks}
	Suppose $X$ is a scheme, and suppose that either $\Lambda$ is integral
	and $X$ is qcqs, or that every constructible closed subset of $X$ has
	Zariski-locally finitely many irreducible components.
	Then the stalks $\mathcal F\mapsto\mathcal F_{\overline x}$, as $x$ runs
	over the points of $X$, form a conservative family of functors for
	$\mathcal D(X,\Lambda)$. We have 
	$\mathcal F_{\overline x}=\Gamma(\overline x,\overline x^*\mathcal F)\in 
	\operatorname{Perf}_\Lambda\subset\mathcal D(\Lambda)$, the full
	subcategory on the perfect complexes.
\end{lemma*}
\begin{proof}
	It follows from the w-contractibility of $\overline x$ and the
	definition of the functor $\overline x^*$ that
	$\Gamma(\overline x,\overline x^*\mathcal F)=\colim\Gamma(U,\mathcal F)$,
	where the (filtered) colimit is over all pro-\'etale neighborhoods $U$ 
	of $\overline x$.\footnote{This relies on the t-exactness of 
	$\overline x^*:D(X_\proet,\Lambda)\to D(\overline x,\Lambda)$ for the
	usual ($p=0$) t-structure. Right t-exactness is clear as
	$H^0\overline x^*:D(X_\proet,\Lambda)^\heartsuit\to
	D(\overline x_\proet,\Lambda)^\heartsuit$ is a left adjoint, and left
	t-exactness follows as $\Gamma(U,-)$ (for $U\in X_\proet$) is left
	t-exact on $D(X_\proet,\Lambda)$ and $H^0\overline x^*$ on 
	$D(X_\proet,\Lambda)^\heartsuit$ returns the sheafification of the 
	presheaf which on a condensed set returns a filtered colimit of such
	$H^0\Gamma(U,-)$.}
	This coincides with $\Gamma(X_{\overline x},\mathcal F)$ by
	\cite[\href{https://stacks.math.columbia.edu/tag/0993}{\texttt{0993}}]{Stacks}.
	So far this is true for any $\mathcal F\in\mathcal D(X_\proet,\Lambda)$,
	but if $\mathcal F\in\mathcal D(X,\Lambda)$, this is a perfect complex
	since
	$\mathcal D_\lis(\overline x,\Lambda)\simeq\operatorname{Perf}_\Lambda$
	\cite[Lemma 1.2]{HRS}.

	Conservativity amounts to showing that a constructible sheaf 
	$\mathcal F$ with zero stalks is zero. Replacing $X$ by an affine open,
	it will suffice to show that its restriction to each constructible
	stratum $Z$ in a stratification witnessing the constructibility of 
	$\mathcal F$ is zero. If $Z$ has locally a finite number of irreducible
	components, then the restriction of $\mathcal F$ to $Z$ is
	(pro-\'etale)-locally perfect-constant \cite[Theorem 4.13]{HRS}.
	If $\overline x$ is a geometric point of $Z$ and $U\to Z$ is a 
	pro-\'etale neighborhood of $\overline x$ so that 
	$\mathcal F|_U\simeq \underline N\otimes_{\underline\Lambda}\Lambda_U$
	for some $N\in\operatorname{Perf}_\Lambda$, then
	$\Gamma(\overline x,\overline x^*\mathcal F)=N$, whence necessarily
	$N=0$.
	
	If $\Lambda$ is integral and $X$ is qcqs, then $\mathcal F$ comes from
	$\mathcal D(X,\mathcal O_E)$ for $E/\QQ_\ell$ finite, so we may assume
	$\mathcal F$ is derived $\ell$-complete. Reduction modulo $\ell$
	commutes with taking stalks ($\ZZ/\ell$ is perfect), so if $\mathcal F$
	has zero stalks the same is true of $\mathcal F/\ell$.
	As $\mathcal F/\ell$ is an \'etale sheaf, this implies its vanishing on
	$X$, so $\mathcal F=0$ by Lemma~\ref{sec:cons_facts}.
\end{proof}
\subsection{}\label{sec:vanishing_cycles_mod_ell}
The nearby and vanishing cycles are constructible
\cite[Theorem 4.1]{HS}, and their formation commutes with
reduction mod. $\ell$ \cite[Lemma 6.5.9(3)]{BS}.
\begin{lemma*}\label{lem:vanishing_cycles_mod_ell}
	Let $V$ be a strictly henselian dvr with $\ell\in V^\times$,
	$f:X\to\Spec V$ a separated map of finite type and
	$\mathcal F\in D(X)$ with $\Lambda=\mathcal O_E$, $E/\QQ_\ell$ finite. 
	Then $\psi_f(\mathcal F)\otimes\ZZ/\ell=\psi_f(\mathcal F/\ell)$ and
	$\phi_f(\mathcal F)\otimes\ZZ/\ell=\phi_f(\mathcal F/\ell)$ are \'etale
	perfect-constructible on the special fiber.
\end{lemma*}
\subsection{}
The usual ($p=0$) t-structure on $D(X_\proet,\Lambda)$ restricts to
one on $D(X)$ under a mild finiteness assumption.

\begin{proposition-cite*}[{\cite[Theorem 6.2]{HRS}}]\label{prop:t-structure}
	Suppose that $X$ is a scheme with the property that every constructible
	subset of $X$ has locally finitely many irreducible components. 
	Then the usual t-structure on $D(X_\proet,\Lambda)$ restricts to a 
	bounded t-structure on $D(X)$. The heart $D(X)^\heartsuit$ consists of
	the full subcategory of $D(X_\proet,\Lambda)^\heartsuit$ on those 
	sheaves $\mathcal G$ with the property that for every open affine
	$U\subset X$ there is a finite stratification of $U$ into constructible
	locally closed subsets $U_i$ so that $\mathcal G|_{U_i}$ is locally on 
	$(U_i)_\proet$ isomorphic to 
	$\underline M_i\otimes_{\underline\Lambda}\Lambda$ for $M_i$ a 
	finitely-presented $\Lambda$-module.
\end{proposition-cite*}
In this case, lisse sheaves admit a familiar description.
\begin{corollary-cite*}[{\cite[Corollary 6.3]{HRS}}]\label{cor:lisse_t}
	Suppose that $X$ is a qcqs scheme with locally finitely many
	irreducible components. Then $M$ in $D(X_\proet,\Lambda)$ is lisse if and only if
	$M$ is bounded and each $\mathcal H^iM$ is locally on $X_\proet$
	isomorphic to $\underline N\otimes_{\underline\Lambda}\Lambda$ for some 
	finitely-presented $\Lambda$ module $N$.
\end{corollary-cite*}
\subsection{}
We now turn to some facts that require additional hypotheses of
regularity.
\begin{proposition*}\label{prop:lisse_equiv}
	Assume $X$ is a topologically noetherian and geometrically unibranch
	scheme, and $E/\QQ_\ell$ an algebraic extension.
	Then $\mathcal D_{\lis}(X,\mathcal O_E)\otimes_{\mathcal O_E}E\to \mathcal D_{\lis}(X,E)$
	is an equivalence.
\end{proposition*}
\begin{proof}
Only essential surjectivity needs to be checked; i.e. the existence of
integral models. We may assume $E/\QQ_\ell$ is finite and $X$ connected with
geometric point $x\to X$. In this case,
$\pi_1^{\proet}(X,x)\simeq\pi_1^{\text{\'et}}(X,x)$
as $X$ is geometrically unibranch \cite[Lemma 7.4.10]{BS}, so essential
surjectivity is clear in degree zero; i.e. for abelian sheaves. Indeed, an
$E$-local system on $X$ is classified by a continuous homomorphism
$\pi_1^\proet(X,x)\to\Aut E$, and the \'etale fundamental group is
profinite, hence compact, hence its action on a finite $E$-vector space
stabilizes a lattice. As any lisse sheaf is bounded, we can induct on
amplitude, using that $\mathcal D_\lis(X,\Lambda)$ is stable.
(In more words, given a map $f:\mathcal F\to\mathcal G$ of lisse 
$E$-sheaves with integral models $\mathcal F_0,\mathcal G_0$, if $n$ is
large enough, $\ell^nf$ is the rationalization of some map 
$f_0:\mathcal F_0\to\mathcal G_0$ in $D(X,\mathcal O_E)$, and
$\Cone\ell^nf\simeq\Cone f$.)
\end{proof}
\subsection{}
The same is true for perverse sheaves, although this won't be used in the
sequel.
It's easy to see that if $X$ is a separated variety and $E/\QQ_\ell$ is an
algebraic extension, the functor 
$-\otimes_{\mathcal O_E}E:\mathcal D(X,\mathcal O_E)\to\mathcal D(X,E)$ is
perverse t-exact. Let $\Perv(X,?)$ denote the heart of $\mathcal D(X,?)$
for the middle perverse t-structure.
\begin{proposition*}
	Let $X$ be a separated variety and $E/\QQ_\ell$ an algebraic extension.
	Then $\Perv(X,\mathcal O_E)\otimes_{\mathcal O_E}E\to\Perv(X,E)$ is an
	equivalence.
\end{proposition*}
\begin{proof}
	Again, only essential surjectivity needs to be checked, so we may assume
	$E/\QQ_\ell$ is finite. Proceeding
	along a Jordan-H\"older series for some $\mathcal F$ in 
	$\Perv(X,E)$, we may assume $\mathcal F$ is (up to a shift)
	the intermediate extension of some $E$-local system $\mathcal L$ on an
	irreducible subvariety $Y\subset X$ with the property that the reduction 
	$(Y_{\alg k})_{\mathrm{red}}$ is smooth over $\alg k$.
	In particular, $Y$ is geometrically unibranch,\footnote{Pick a geometric
	point $y\to Y$; then the strict localization $Y_y$ of $Y$ at $y$
	coincides with that of $Y_{\sep k}$ at $y$. The extension
	$\alg k/\sep k$ being purely inseparable, the base change
	$(Y_y)_{\alg k}$ is homeomorphic to $Y_y$ and coincides with the strict
	localization $(Y_{\alg k})_y$ of $Y_{\alg k}$ at $y$ (as the spectrum of 
	a filtered colimit of strictly henselian local rings along local
	homomorphisms). Quotienting $Y_{\alg k}$ by nilpotents it becomes 
	smooth, so $(Y_{\alg k})_y$ has a unique minimal prime and we're done.}
	so we may find a lattice for $\mathcal L$. It remains only to show that
	$-\otimes\QQ_\ell$ commutes with intermediate extension. This is
	immediate from Deligne's formula \cite[Proposition 2.2.4]{BBD} in light
	of the fact that $j_*$ and $\tau^F_{\leq0}$ (with the notation there)
	commute with $-\otimes\QQ_\ell$, the former by \cite[Lemma 5.3]{HRS}
	and the latter by the formula
	$\tau_{\leq0}^F\mathcal G=\fib(\mathcal G\to i_*\tau_{>0}i^*\mathcal G)$
	and the t-exactness of $-\otimes\QQ_\ell$ for the usual ($p=0$)
	t-structure.
\end{proof}

\section{Universal local acyclicity}\label{sec:ula}
This section is dedicated to a discussion of the notion of universal local
acyclicity we will use to define the singular support of a constructible 
$\ell$-adic sheaf on a smooth variety. Almost all of the ideas in this
section are either implicit or explicit in \cite{HS}, so there is no claim
of originality.
Without further remark, all schemes in this section are assumed to live
over $\Spec\ZZ[\frac1\ell]$.

While the classical definition of local acyclicity is phrased in terms of
Milnor fibers, Lu and Zheng discovered in \cite{LZ_ula} a 
characterization of universal local acyclicity that remains in the world of
varieties. Their characterization is in terms of dualizability in a certain
symmetric monoidal 2-category (bicategory) of cohomological
correspondances. We now recall the definition of that 2-category, or rather
of a slight variation on it with the same dualizable objects that is due to
\cite{HS}.

\subsection{}\label{sec:C_S}
In the following discussion (i.e. until \S\ref{sec:def_local_acyclicity}),
all schemes are moreover assumed qcqs. Fix a base scheme $S$.
Our 2-category $\mathcal C_S:=\mathcal C_{S,\Lambda}$ has objects given by
pairs $(X,\mathcal F)$, where $f:X\to S$ is separated of finite 
presentation and $\mathcal F\in D(X)$.
A 1-morphism $g:(X,\mathcal F)\to(Y,\mathcal G)$ in $\mathcal C_S$ is the
data of some $\mathcal K\in D(X\times_SY)$ together with a map
$\delta:\pi_{Y!}(\pi_X^*\mathcal F\otimes\mathcal K)\to\mathcal G$,
where $\pi_X,\pi_Y$ denote the projections 
$X\times_SY\to X,Y$.
Given 1-morphisms $(\mathcal K,\delta)$ and 
$(\mathcal M,\vep):(X,\mathcal F)\to(Y,\mathcal G)$, a 2-morphism
$(\mathcal K,\delta)\Rightarrow(\mathcal M,\vep)$ is given by a map
$\mathcal K\to\mathcal M$ in $D(X\times_SY)$ making a commutative triangle
\begin{equation*}\begin{tikzcd}[row sep=2]
	\pi_{Y!}(\pi_X^*\mathcal F\otimes\mathcal K)\arrow[dr,"\delta"]\arrow[dd] \\
	&\mathcal G. \\
	\pi_{Y!}(\pi_X^*\mathcal F\otimes\mathcal M)\arrow[ur,"\vep"']
\end{tikzcd}\end{equation*}

Concerning composition in $\mathcal C_S$, suppose given 1-morphisms
$(X,\mathcal F)\xra{\mathcal L}(Y,\mathcal G)\xra{\mathcal M}(Z,\mathcal K)$
in $\mathcal C_S$
($\mathcal L\in D(X\times_SY)$, $\mathcal M\in D(Y\times_SZ)$), which come with maps
$\delta:\pi_{Y!}(\pi_X^*\mathcal F\otimes\mathcal L)\to\mathcal G$
and $\vep:\pi_{Z!}(\pi_Y^*\mathcal G\otimes\mathcal M)\to\mathcal K$.
The composite $(\mathcal M,\vep)\circ(\mathcal L,\delta)$ is given by
$\pi_{XZ!}(\pi_{XY}^*\mathcal L\otimes\pi_{YZ}^*\mathcal M)
\in D(X\times_SZ)$,\footnote{Here, $\pi_{XZ}$ denotes the projection
$X\times_SY\times_SZ\to X\times_SZ$, etc.} together with a map 
$\pi_{Z!}(\pi_X^*\mathcal F\otimes\pi_{XZ!}(\pi_{XY}^*\mathcal L\otimes\pi_{YZ}^*\mathcal M)
\to\mathcal K$ obtained using proper base change and the projection 
formula as follows.\footnote{In light of the existence of integral models,
the validity of these formulas in this context, as well as that of
the K\"unneth formula used below, follows directly from the corresponding 
facts in \'etale cohomology as the relevant functors on integral sheaves
commute with rationalization and reduction modulo $\ell$.
See \cite[Lemma 6.7.10 \& Lemma 6.7.14]{BS}}
(In the following, the projections $\pi$ denote projections from 
$X\times_SY$, $X\times_SY\times_SZ$ and the like, and the meaning of the
notation changes from line to line.)
\begin{align*}
	&\pi_{Z!}(\pi_X^*\mathcal F\otimes\pi_{XZ!}(\pi_{XY}^*\mathcal L\otimes\pi_{YZ}^*\mathcal M)&(\pi_Z:X\times_SZ\to Z) \\
	=&\ \pi_{Z!}(\pi_X^*\mathcal F\otimes\pi_{XY}^*\mathcal L\otimes\pi_{YZ}^*\mathcal M)
	&(\pi_Z:X\times_SY\times_SZ\to Z) \\
	=&\ \pi_{Z!}(\pi_{YZ!}(\pi_X^*\mathcal F\otimes\pi_{XY}^*\mathcal L)\otimes\mathcal M)
	&(\pi_Z:Y\times_SZ\to Z) \\
	=&\ \pi_{Z!}(\pi_Y^*(\pi'_{Y!}(\pi_X^*\mathcal F\otimes\mathcal L))\otimes\mathcal M)
	&(\pi_Y:Y\times_SZ\to Y,\pi_Y':X\times_SY\to Y) \\
	\xra\delta&\ \pi_{Z!}(\pi_Y^*\mathcal G\otimes\mathcal M)
	\xra\vep\mathcal K.
\end{align*}

The symmetric monoidal structure on $\mathcal C_S$ works as follows: on
objects we have 
$(X,\mathcal F)\boxtimes(Y,\mathcal G)
:=(X\times_SY,\mathcal F\boxtimes\mathcal G:=\pi_X^*\mathcal F\otimes\pi_Y^*\mathcal G)$.
On 1-morphisms, suppose given 1-morphisms $(X,\mathcal F)\to(Y,\mathcal G)$ 
and $(X',\mathcal F')\to(Y',\mathcal G')$ represented by objects
$\mathcal K\in D(X\times_SY)$ and $\mathcal K'\in D(X'\times_SY')$ and 
arrows $\pi_{Y!}(\pi_X^*\mathcal F\otimes\mathcal K)\to\mathcal G$ and
$\pi_{Y'!}(\pi_{X'}^*\mathcal F'\otimes\mathcal K')\to\mathcal G'$. Their
tensor product is the 1-morphism 
$(X\times_SX',\mathcal F\boxtimes\mathcal F')\to(Y\times_SY',\mathcal G\boxtimes\mathcal G')$
given by
$\mathcal K\boxtimes\mathcal K'\in D(X\times_SY\times_SX'\times_SY')$ 
together with the morphism
$\pi_{Y!}(\pi_X^*(\mathcal F\boxtimes\mathcal F')\otimes(\mathcal K\boxtimes\mathcal K'))=
\pi_{Y!}((\pi_X'^*\mathcal F\otimes\mathcal K)\boxtimes(\pi_{X'}'^*\mathcal F'\otimes\mathcal K'))
=\pi_{Y!}'(\pi_X'^*\mathcal F\otimes\mathcal K)\boxtimes\pi_{Y'!}'(\pi_{X'}'^*\mathcal F'\otimes\mathcal K')
\to\mathcal G\boxtimes\mathcal G'$,
where here $\pi_X:X\times_SX'\times_SY\times_SY'\to X\times_SX'$,
$\pi_X':X\times_SY\to X$ (likewise for $\pi_Y$), and the second isomorphism 
is given by K\"unneth.

On 2-morphisms, given 1-morphisms $f_1,f_2:(X,\mathcal F)\to(Y,\mathcal G)$
represented by $\mathcal K_1,\mathcal K_2\in D(X\times_SY)$ and 1-morphisms
$g_1,g_2:(X',\mathcal F')\to(Y',\mathcal G')$ 
represented by $\mathcal K_1',\mathcal K_2'\in D(X'\times_SY')$, and given
2-morphisms $\alpha:f_1\Rightarrow f_2$ represented by some arrow 
$\mathcal K_1\to\mathcal K_2$ in $D(X\times_SY)$ and $\beta:g_1\Rightarrow g_2$ given by $\mathcal K_1'\to\mathcal K_2'$,
$\alpha\boxtimes\beta:f_1\boxtimes g_1\Rightarrow f_2\boxtimes g_2$ is
given by
$\mathcal K_1\boxtimes\mathcal K_1'\to\mathcal K_2\boxtimes\mathcal K_2'$, 
which satisfies the required compatibility in light of the K\"unneth
isomorphism above.

The monoidal unit $1_{\mathcal C_S}$ is given by $(S,\Lambda)$.

Given a morphism of schemes $f:S'\to S$, base change gives a symmetric
monoidal functor $f^*:\mathcal C_S\to\mathcal C_{S'}$ which on objects
sends $(X,\mathcal F)$ to $(X\times_SS',f^*\mathcal F)$. That this
determines a symmetric monoidal functor relies on proper base change.

If $\Lambda=\mathcal O_E$ is integral, reduction modulo $\ell$ and
inverting $\ell$ determine symmetric monoidal functors 
$\mathcal C_{S,\mathcal O_E}\to\mathcal C_{S,\mathcal O_E/\ell}$ and
$\mathcal C_{S,\mathcal O_E}\to\mathcal C_{S,E}$, respectively
(we will only consider the former functor when $E/\QQ_\ell$ is finite).
On objects these functors send $(X,\mathcal F)$ to $(X,\mathcal F/\ell)$
and $(X,\mathcal F[\ell^{-1}]=\mathcal F\otimes_{\mathcal O_E}E)$, 
respectively. That these determine symmetric monoidal functors rests on
the commutation of pullback and direct image with compact support with
reduction modulo $\ell$ and rationalization; for rationalization this
follows from the fact that direct image between qcqs schemes commutes with
filtered colimits in $D^{\geq 0}(?_\proet,\Lambda)$.

\subsection{}\label{sec:duals_in_2cats}
An object $V$ of a symmetric monoidal 2-category (bicategory)
$(\mathcal C,\otimes,1_\mathcal C)$ is said to be \emph{dualizable} if
there exists an object $V^\vee$ of $\mathcal C$, called the \emph{dual} of
$V$, and 1-morphisms $\ev:V^\vee\otimes V\to1_{\mathcal C}$,
$\coev:1_{\mathcal C}\to V\otimes V^\vee$ so that the composites
\begin{equation*}
	V\xrightarrow{\coev\otimes\id} V\otimes V^\vee\otimes V\xra{\id\otimes\ev} V
	\quad\text{and}\quad
	V^\vee\xra{\id\otimes\coev} V^\vee\otimes V\otimes V^\vee\xra{\ev\otimes\id}V^\vee
\end{equation*}
are isomorphic to the respective identities.
It's clear from the definition that if $V$ is dualizable, then $V^\vee$ is
too and $V^{\vee\vee}=V$. Likewise, if $V$ and $W$ are dualizable, then
$V\otimes W$ is dualizable with dual $V^\vee\otimes W^\vee$. If
$F:(\mathcal C,\otimes,1_\mathcal C)\to(\mathcal D,\otimes,1_{\mathcal D})$
is a symmetric monoidal functor of symmetric monoidal 2-categories, then
$F$ preserves dualizable objects and duals.

If $V$ is dualizable, then the morphisms $\ev$ and $\coev$ exhibit
$-\otimes V^\vee$ as right (and left) adjoint to $-\otimes V$.
Thus, for every object $W$ of $\mathcal C$, the internal Hom
$\sHom(V,W)$ exists and is equivalent to $W\otimes V^\vee$.
Conversely, dualizability can be characterized in terms of certain internal
mapping objects.
\begin{lemma-cite*}[{\cite[Lemma 1.4]{LZ_ula}}]
	An object $V$ of $\mathcal C$ is dualizable if and only if the internal Hom objects
	$\sHom(V,1_{\mathcal C})$ and $\sHom(V,V)$ exist and the morphism
	$V\otimes\sHom(V,1_\mathcal C)\to\sHom(V,V)$ adjoint to the map
	$V\otimes\sHom(V,1_\mathcal C)\otimes V\xra{\id\otimes\ev} V$ is a
	split epimorphism, where
	$\ev:\sHom(V,1_{\mathcal C})\otimes V\to1_{\mathcal C}$
	denotes the counit of adjunction.
\end{lemma-cite*}

\subsection{} Fix a base scheme $S$.
Given a dualizable object $(X,\mathcal F)$ of $\mathcal C_S$, the dual of
$(X,\mathcal F)$ has a concrete description in terms of the relative
Verdier dual. Let 
\begin{equation*}
	\mathbf D_{X/S}(\mathcal F):=
	\sHom_{\mathcal D(X_\proet,\Lambda)}(\mathcal F,f^!\Lambda)
	\in\mathcal D(X_\proet,\Lambda).
\end{equation*}
Here, if $\Lambda=\mathcal O_E$ with $E/\QQ_\ell$ a finite extension,
$f^!\Lambda:=\lim_nf^!(\mathcal O_E/\ell^n)$ as an object of
$\mathcal D(X_\proet,\Lambda)$. Here, $\mathcal O_E/\ell^n$ is an \'etale 
sheaf, so the functor $f^!$ is defined, but $f^!(\mathcal O_E/\ell^n)$ 
needn't be constructible.

If $\Lambda=E$ with $E/\QQ_\ell$ finite,
$f^!\Lambda:=(\lim_nf^!(\mathcal O_E/\ell^n))\otimes_{\ZZ_\ell}\QQ_\ell
=(\lim_nf^!(\mathcal O_E/\ell^n))\otimes_{\mathcal O_E}E
=\colim((\lim_nf^!(\mathcal O_E/\ell^n))\xra\ell(\lim_nf^!(\mathcal O_E/\ell^n))\xra\ell\ldots)=:\colim_\ell\lim_n(f^!(\mathcal O_E/\ell^n))$,
where these limits and colimits are taken in
$\mathcal D(X_\proet,\mathcal O_E)$.

If $\mathcal F\in\mathcal D(X,\Lambda)$ with $\Lambda=\mathcal O_E$ and
$E/\QQ_\ell$ infinite algebraic, we may find a 
$\mathcal F'\in\mathcal D(X,\mathcal O_L)$, $L/\QQ_\ell$ finite, so that 
$\mathcal F=\mathcal F'\otimes_{\mathcal O_L}\mathcal O_E$, and then
$f^!\Lambda:=(f^!\mathcal O_L)\otimes_{\mathcal O_L}\mathcal O_E$, so that
$\mathbf D_{X/S}(\mathcal F)=\mathbf D_{X/S}(\mathcal F')\otimes_{\mathcal O_L}\mathcal O_E$
\cite[Lemma 5.3]{HRS}.

If $\Lambda=E$ with $E/\QQ_\ell$ infinite algebraic, we may again 
approximate $\mathcal F$ by $\mathcal F'\in\mathcal D(X,L)$ and set
$f^!\Lambda:=(f^!\mathcal O_E)\otimes_{\mathcal O_E}E
=(f^!\mathcal O_L)\otimes_{\mathcal O_L}E
=(f^!L)\otimes_LE$,
so that
$\mathbf D_{X/S}(\mathcal F)=\mathbf D_{X/S}(\mathcal F')\otimes_LE$.

\begin{proposition-cite*}[{\cite[Proposition 3.4]{HS}}]\label{prop:duals_in_C_S}
	Let $f:X\to S$ be a separated map of finite presentation and 
	$(X,\mathcal F)$ be dualizable in $\mathcal C_S$. Then the relative
	Verdier dual $\mathbf D_{X/S}(\mathcal F)$ is constructible and
	$(X,\mathbf D_{X/S}(\mathcal F))$ is the dual of $(X,\mathcal F)$ in 
	$\mathcal C_S$. In particular, the biduality map 
	$\mathcal F\to\mathbf D_{X/S}(\mathbf D_{X/S}(\mathcal F))$
	is an isomorphism, and the formation of $\mathbf D_{X/S}(\mathcal F)$
	commutes with any base change in $S$, as well as with reduction modulo
	$\ell$ if $\Lambda=\mathcal O_E$ with $E/\QQ_\ell$ finite.
\end{proposition-cite*}

\subsection{}\label{sec:C_S_internal_homs}
In view of Lemma~\ref{sec:duals_in_2cats}, it's useful to understand the
internal mapping objects in $\mathcal C_S$.
\begin{proposition*}\label{HS_3.3_internal_hom}
Fix a finite extension $L$ of $\QQ_\ell$.
Suppose $(X,\mathcal F)$ and $(Y,\mathcal G)$ are in $\mathcal C_S$ with 
$\mathcal O_L$-coefficients. Then the internal Hom
$\sHom_{\mathcal C_S}((X,\mathcal F),(Y,\mathcal G))$ exists and is given 
by
\begin{equation*}(X\times_SY,\sHom_{\proet}(\pi_1^*\mathcal F,\lim_n\pi_2^!(\mathcal G/\ell^n))\end{equation*}
if this sheaf is constructible.

Similarly, if $(X,\mathcal F)$ and $(Y,\mathcal G)$ are in $\mathcal C_S$
with $L$-coefficients, then the internal Hom 
$\sHom_{\mathcal C_S}((X,\mathcal F),(Y,\mathcal G))$ exists and is given 
by
\begin{equation*}(X\times_SY,\sHom_{\proet}(\pi_1^*\mathcal F,(\lim_n\pi_2^!(\mathcal G_0/\ell^n))[\ell^{-1}])\end{equation*}
if this sheaf is constructible, where $\mathcal G_0$ is any integral model
for $\mathcal G$.
\end{proposition*}
\begin{proof*}
Let's first consider integral coefficients. To check the claim, we have to
check that adjunction
\begin{multline}\label{eq:hom_categories}\tag{$\ast$}
\Hom_{\mathcal C_S}((X,\mathcal F)\otimes(Y,\mathcal G),(Z,\mathcal M)) \\\cong
\Hom_{\mathcal C_S}((X,\mathcal F),(Y\times_SZ,\sHom(\pi_Y^*\mathcal G,\lim_n\pi_Z^!(\mathcal M/\ell^n))))\end{multline}
determines an isomorphism of categories.
Objects of the first category are objects 
$\mathcal K\in D(X\times_SY\times_SZ)$
equipped with an arrow
$\pi_{Z!}(\pi_{XY}^*(\mathcal F\boxtimes \mathcal G)\otimes\mathcal K)\to\mathcal M$
in $D(Z)$. Objects of the latter category are objects
$\mathcal K\in D(X\times_SY\times_SZ)$ equipped with an arrow
$\pi_{YZ!}(\pi_X^*\mathcal F\otimes\mathcal K)\to\sHom(\pi_Y^*\mathcal G,\lim_n\pi_Z^!(\mathcal M/\ell^n))$
(where $\pi_{YZ}:X\times_SY\times_SZ\to Y\times_SZ$ is the projection),
i.e. with an arrow $\pi_{YZ!}(\pi_X^*\mathcal F\otimes\mathcal K)\otimes\pi_Y^*\mathcal G
=\pi_{YZ!}(\pi_{XY}^*(\mathcal F\boxtimes\mathcal G)\otimes\mathcal K)\to\lim_n\pi_Z^!(\mathcal M/\ell^n)$.
Given a separated finitely-presented map $h:Y\to X$ and constructible
sheaves $\mathcal A$ and $\mathcal B$ on $Y$ and $X$, respectively, there
is an isomorphism of Hom complexes
$\Hom(\mathcal A,\lim_n h^!(\mathcal B/\ell^n))
=\Hom(h_!\mathcal A,\mathcal B)$
in $\mathcal D(\Lambda)$ obtained by passing the limit outside and using
the adjunction on the level of \'etale sheaves.
Therefore the objects in the categories~\eqref{eq:hom_categories} are the
same.

Arrows of the first category are maps $\mathcal K\to\mathcal H$ in
$D(X\times_SY\times_SZ)$ making the triangle below commute.
\begin{equation*}\begin{tikzcd}
	\pi_{Z!}(\pi_{XY}^*(\mathcal F\boxtimes\mathcal G)\otimes\mathcal K)\arrow[d]\arrow[r]&\mathcal M \\
	\pi_{Z!}(\pi_{XY}^*(\mathcal F\boxtimes\mathcal G)\otimes\mathcal H)\arrow[ur]
\end{tikzcd}
\end{equation*}
Arrows of the second category are maps $\mathcal K\to\mathcal H$ making the
triangle below commute.
\begin{equation*}\begin{tikzcd}
	\pi_{YZ!}(\pi_X^*\mathcal F\otimes\mathcal K)\arrow[r]\arrow[d]&
	\sHom(\pi_Y^*\mathcal G,\lim_n\pi_Z^!(\mathcal M/\ell^n)) \\
	\pi_{YZ!}(\pi_X^*\mathcal F\otimes\mathcal H)\arrow[ur]
\end{tikzcd}\end{equation*}
These are seen to be the same under adjunction;
therefore the two $\Hom$ categories are indeed isomorphic.

With $L$-coefficients the argument is the same using also
\cite[Lemma 5.3]{HRS}. With $\mathcal A_0$ and $\mathcal B_0$ in
$\mathcal D(Y,\mathcal O_L)$ and $\mathcal D(X,\mathcal O_L)$,
respectively, 
$\mathcal A=\mathcal A_0[\ell^{-1}]$ etc. and $h$ as above, we have 
isomorphisms in $\mathcal D(L)$ of the type
\begin{align*}
	\Hom(\mathcal A,(\lim_n h^!(\mathcal B_0/\ell^n))[\ell^{-1}])
	&=\colim_\ell\lim_n\Hom(\mathcal A_0/\ell^n,h^!(\mathcal B_0/\ell^n)) \\
	&=\colim_\ell\lim_n\Hom(h_!(\mathcal A_0/\ell^n),\mathcal B_0/\ell^n) \\
	&=\Hom(h_!\mathcal A_0,\mathcal B)=\Hom(h_!\mathcal A,\mathcal B).
	&\pushright\qed
\end{align*}
\end{proof*}
The following corollary is the $\ell$-adic analogue of
\cite[Proposition 3.3]{HS} and follows from Lemma~\ref{sec:duals_in_2cats}
in light of the previous result.
\begin{corollary*}\label{HS:3.3_corollary}
	Suppose $f:X\to S$ is a separated finitely-presented morphism of schemes
	and $(X,\mathcal F)$ is in $\mathcal C_S$ with $\Lambda=L$ or
	$\mathcal O_L$, $L$ a finite extension of $\QQ_\ell$.
	Then $(X,\mathcal F)$ is dualizable if
	$\mathbf D_{X/S}(\mathcal F)=\sHom(\mathcal F,f^!\Lambda)$
	is constructible and the natural map
	\begin{equation*}
		\mathbf D_{X/S}(\mathcal F)\boxtimes_S\mathcal F
		\lra\sHom(\pi_1^*\mathcal F,\pi_2^!\mathcal F)
	\end{equation*}
	is an isomorphism in $D(X\times_SX)$, where as usual $\pi_2^!\mathcal F$
	denotes the pro-\'etale sheaf 
	$\lim_n\pi_2^!(\mathcal F/\ell^n)$
	with integral coefficients or 
	$(\lim_n\pi_2^!(\mathcal F_0/\ell^n))[\ell^{-1}]$
	with $L$-coefficients, where $\mathcal F_0$ is any integral model for
	$\mathcal F$, and likewise for $f^!\Lambda$.
\end{corollary*}
\subsection{}\label{sec:converse_internal_hom}
The following in particular provides a converse to the previous result.
Together, these results appear in the second paragraph following
\cite[Proposition 3.3]{HS}, and the proof below implements the approach
sketched there.
\begin{proposition*}\label{prop:HS_3.3_dualizable_internal_hom}
	Suppose $L/\QQ_\ell$ is finite and $(X,\mathcal F)$ is dualizable in
	$\mathcal C_S$ with $\mathcal O_L$-coefficients.
	Then for every $(Y,\mathcal G)$ in $\mathcal C_S$,
	$\sHom(\pi_1^*\mathcal F,\lim_n\pi_2^!(\mathcal G/\ell^n))$
	is constructible on $X\times_SY$ and we have via the canonical map
	an isomorphism
	\begin{equation*}
		\mathbf D_{X/S}(\mathcal F)\boxtimes_S\mathcal G\xlongrightarrow\sim
		\sHom(\pi_1^*\mathcal F,\lim_n\pi_2^!(\mathcal G/\ell^n)).
	\end{equation*}
	
	Likewise, if $(X,\mathcal F)$ is dualizable in $\mathcal C_S$ with
	$L$-coefficients, then for every $(Y,\mathcal G)$ in $\mathcal C_S$,
	$\sHom(\pi_1^*\mathcal F,(\lim_n\pi_2^!\mathcal G_n)[\ell^{-1}])$
	is constructible on $X\times_SY$ and isomorphic to
	$\mathbf D_{X/S}(\mathcal F)\boxtimes_S\mathcal G$ via the canonical 
	map, where $\mathcal G_n:=\mathcal G_0/\ell^n$ for
	$\mathcal G_0$ any integral model for $\mathcal G$.
\end{proposition*}
\begin{proof}
Assume first $\Lambda$ is integral.
The dual of $(X,\mathcal F)$ is $(X,\mathbf D_{X/S}\mathcal F)$. Then the 
point is that $(X,\mathbf D_{X/S}(\mathcal F))\otimes(Y,\mathcal G)
=(X\times_SY,\mathbf D_{X/S}(\mathcal F)\boxtimes_S\mathcal G)$ is 
an internal Hom in $\mathcal C_S$. As
$\mathbf D_{X/S}(\mathcal F)\boxtimes_S\mathcal G$ is constructible,
$\mathbf D_{X/S}(\mathcal F)\boxtimes_S\mathcal G=\lim_n((\mathbf D_{X/S}(\mathcal F)\boxtimes_S\mathcal G)\otimes_{\ZZ_\ell}\ZZ_\ell/\ell^n)$.
As reduction modulo $\ell^n$ is a symmetric monoidal functor on the
categories $\mathcal C_S$, it preserves internal Homs; therefore
$(\mathbf D_{X/S}(\mathcal F)\boxtimes_S\mathcal G)\otimes_{\ZZ_\ell}\ZZ_\ell/\ell^n=
\sHom(\pi_1^*(\mathcal F/\ell^n),\pi_2^!(\mathcal G/\ell^n))
=\sHom(\pi_1^*\mathcal F,\pi_2^!(\mathcal G/\ell^n))$,\footnote{This uses the fact that if you build the category
$\mathcal C_S$ with $D(?)$ replaced by the derived category of (not
necessarily constructible) \'etale sheaves of
$\mathcal O_L/\ell^n$-modules on $?\to S$ separated and finitely-presented,
then internal Homs always exist and are given by
$\Hom((X,\mathcal F),(Y,\mathcal G))=(X\times_SY,\sHom(\pi_X^*\mathcal F,\pi_Y^!\mathcal G))$.
The argument is the same as the one appearing in the proof of
Lemma~\ref{sec:C_S_internal_homs}, but simpler.} and the sheaf
$\mathbf D_{X/S}(\mathcal F)\boxtimes_S\mathcal G
=\lim_n\sHom(\pi_1^*\mathcal F,\pi_2^!(\mathcal G/\ell^n))
=\sHom(\pi_1^*\mathcal F,\lim_n\pi_2^!(\mathcal G/\ell^n))$
is constructible.

Now suppose $\Lambda$ is rational and suppose $(X,\mathcal F)$ is
dualizable. By \cite[Proposition~3.8]{HS}, there is a $v$-cover $S'\to S$
and an integral model $\widetilde{\mathcal F}_0$ for $\mathcal F|_{X'}$
which is universally locally acyclic over $S'$. We also pick integral 
models $\mathcal F_0,\mathcal G_0$ for $\mathcal F,\mathcal G$ on $X$ and 
$Y$ (not necessarily universally locally acyclic) and let 
$\mathcal F_n:=\mathcal F_0/\ell^n$, etc.
We wish to show the natural map
$\mathbf D_{X/S}(\mathcal F)\boxtimes_S\mathcal G\to
\sHom(\pi_1^*\mathcal F,(\lim_n\pi_2^!\mathcal G_n)[\ell^{-1}])$
is an equivalence. We check on sections over $j:U\to X\times_SY$
in $(X\times_SY)_\proet$. The statement is local, so we may assume $X$, $Y$
and $S$ are affine and that $U=\lim_iU_i\to X\times_SY$ is a limit of
affine \'etale maps $j_i:U_i\to X\times_SY$.
We denote by $S'^{\bullet/S}$ the \v Cech nerve associated to $S'\to S$
as in \S\ref{sec:cons_facts}, and let $X'^{\bullet/X}$ denote the \v Cech
nerve associated to the $v$-cover $X':=X\times_SS'\to X$ obtained via base
change, so that $X'^{\bullet/X}=X\times_SS'^{\bullet/S}$, and likewise for
$Y$. Let $p_1,p_2,\pi_1,\pi_2$ be defined by the diagram on the left below
\begin{equation*}\begin{tikzcd}[column sep=0,row sep=15]
	&U\arrow[d,"j"]\arrow[ddl,bend right,"p_1"']\arrow[ddr,bend left,"p_2"] \\
	&X\times_SY\arrow[dl,"\pi_1"]\arrow[dr,"\pi_2"'] \\
	X\arrow[dr]&&Y\arrow[dl] \\ &S,
\end{tikzcd} \qquad \qquad \begin{tikzcd}[column sep=-30]
	&U'^{\bullet/U}=U\times_SS'^{\bullet/S}\arrow[d,"j"]\arrow[ddl,bend right=40,"p_1"'] \arrow[ddr,bend left=40,"p_2"] \\
	&X\times_SY\times_SS'^{\bullet/S}\arrow[dl,"\pi_1"]\arrow[dr,"\pi_2"'] \\
	X'^{\bullet/X}&&Y'^{\bullet/Y},
\end{tikzcd}\end{equation*}
and let $p_{1i},p_{2i}$ defined similarly in terms of $j_i$ in lieu of $j$.
Put $U':=U\times_SS'$ and let $j,p_1,p_2$ continue to describe base
extensions as in the diagram to the right above, and likewise for $j_i,p_{1i},p_{2i}$.

We use the result for integral coefficients and write
\begin{align*}
	\Gamma(U,\mathbf D_{X/S}(\mathcal F)\boxtimes_S\mathcal G)
	&=\lim_\Delta\Gamma(U'^{\bullet/U},
	j^*(\mathbf D_{X'^{\bullet/X}/S'^{\bullet/S}}(\mathcal F)\boxtimes_S\mathcal G)) \\
	&=\lim_\Delta\Gamma(U'^{\bullet/U},
	j^*(\mathbf D_{X'^{\bullet/X}/S'^{\bullet/S}}(\widetilde{\mathcal F}_0)\boxtimes_S\mathcal G_0)\otimes_{\mathcal O_L}L) \\
	&=\lim_\Delta\Gamma(U'^{\bullet/U},
	j^*\sHom(\pi_1^*\widetilde{\mathcal F}_0,\pi_2^!\mathcal G_0)\otimes_{\mathcal O_L}L) \\
	&=\lim_\Delta\Hom_{U'^{\bullet/U}}(j^*\pi_1^*\mathcal F,j^*\pi_2^!\mathcal G) \\
	&=\lim_\Delta\colim_\ell\lim_n\Hom_{U'^{\bullet/U}}(j^*\pi_1^*\mathcal F_n,j^*\pi_2^!\mathcal G_n) \\
	&=\lim_\Delta\colim_\ell\lim_n\Hom_{X\times_SY\times_SS'^{\bullet/S}}(\pi_1^*\mathcal F_n,j_*j^*\pi_2^!\mathcal G_n) \\
	&=\lim_\Delta\colim_\ell\lim_n\colim_i\Hom_{X\times_SY\times_SS'^{\bullet/S}}(\pi_1^*\mathcal F_n,j_{i*}j_i^*\pi_2^!\mathcal G_n) \\
	&=\lim_\Delta\colim_\ell\lim_n\colim_i\Hom_{U_i'^{\bullet/U_i}}(p_{1i}^*\mathcal F_n,p_{2i}^!\mathcal G_n) \\
	&=\colim_\ell\lim_n\colim_i\lim_\Delta\Hom_{Y'^{\bullet/Y}}(p_{2i!}p_{1i}^*\mathcal F_n,\mathcal G_n) \\
	&=\colim_\ell\lim_n\colim_i\Hom_Y(p_{2i!}p_{1i}^*\mathcal F_n,\mathcal G_n) \\
	&=\colim_\ell\lim_n\colim_i\Hom_{X\times_SY}(\pi_1^*\mathcal F_n,j_{i*}j_i^*\pi_2^!\mathcal G_n) \\
	&=\colim_\ell\lim_n\Hom_{X\times_SY}(\pi_1^*\mathcal F_n,j_*j^*\pi_2^!\mathcal G_n) \\
	&=\colim_\ell\lim_n\Hom_U(j^*\pi_1^*\mathcal F_n,j^*\pi_2^!\mathcal G_n) \\
	&=\Hom_U(j^*\pi_1^*\mathcal F,j^*\pi_2^!\mathcal G) \\
	&=\Gamma(U,\sHom_{X\times_SY}(\pi_1^*\mathcal F,(\lim_n\pi_2^!\mathcal G_n)[\ell^{-1}])).
\end{align*}
The point is that we need both arguments of the $\Hom$ to be constructible
to apply $v$-descent. This is true for
$\Hom(p_{2i!}p_{1i}^*\mathcal F_n,\mathcal G_n)$ because $p_{2i}$ is
separated of finite presentation (this is why we pass from $j_*j^*$ to 
$\colim_i j_{i*}j_i^*$).
We can pass the totalization through filtered colimits because the
totalization behaves like a finite limit as 
$\mathcal D^{\geq0}(\mathcal O_L/\ell^n)$ is compactly generated by
cotruncated objects \cite[Lemma 3.7]{BM}.
Finally, the isomorphism
$\colim_ij_{i*}j_i^*\pi_2^!\mathcal G_n\xra\sim j_*j^*\pi_2^!\mathcal G_n$
expresses the finitaryness of \'etale cohomology for bounded below 
complexes
\cite[\href{https://stacks.math.columbia.edu/tag/0GIV}{\texttt{0GIV}}]{Stacks}.
\end{proof}

\subsection{}\label{sec:def_local_acyclicity}
Fix a morphism of schemes $f:X\to S$ locally of finite presentation and a
sheaf $\mathcal F$ in $D(X)$, defining an object $(X,\mathcal F)$ in 
$\mathcal C_S$. If $f$ is separated and $X$ and $S$ are qcqs, following Lu
and Zheng, Hansen and Scholze define $f$ to be universally locally acyclic
relative to $\mathcal F$ if $(X,\mathcal F)$ is dualizable in 
$\mathcal C_S$.
We already made in Definition~\ref{sec:intro_ula} a definition of universal
local acyclicity for $f$ that has a purely local characterization.
This definition specializes to the one of Hansen and Scholze, as is seen 
from the following
\begin{proposition-cite*}[{\cite[Theorem 4.4]{HS}}]\label{prop:HS4.4}
	Let $f:X\to S$ be a locally finitely presented morphism of schemes and 
	let $\mathcal F\in D(X)$. Then the following are equivalent:
	\begin{enumerate}[label=(\roman*)]
	\item\label{4.4_ula_TFAE} The sheaf $\mathcal F$ is universally locally
	acyclic over $S$; i.e. for any geometric point $x\to X$ and generization
	$t\to S$ of $f(x)$, the map
	$\mathcal F_x:=\Gamma(X_x,\mathcal F)\to\Gamma(X_x\times_{S_{f(x)}}t,\mathcal F)$
	is an isomorphism, and this remains true after arbitrary base change in
	$S$.
	\item\label{4.4_cover_TFAE} The scheme $S$ can be covered by qcqs opens
	$\{V_i\}$ with the property that for each $i$, $f^{-1}(V_i)$ can be
	covered by qcqs opens $\{U_{ij}\}$, each separated over $V_i$, such that
	$(U_{ij},\mathcal F|_{U_{ij}})$ defines a dualizable object of 
	$\mathcal C_{V_i}$ for all $i,j$.
	\item\label{4.4_all_cover_TFAE} For every cover of $S$ by qcqs opens 
	$\{V_i\}$ and for every cover of $f^{-1}(V_i)$ by qcqs opens 
	$\{U_{ij}\}$, each separated over $V_i$, $(U_{ij},\mathcal F|_{U_{ij}})$ 
	defines a dualizable object of $\mathcal C_{V_i}$ for all $i,j$.
	\item\label{4.4_affine_TFAE} For every pair of affine opens 
	$U\subset X$, $V\subset S$ with $f(U)\subset V$, $(U,\mathcal F|_U)$
	defines a dualizable object in $\mathcal C_V$.
	\item\label{4.4_milnor_fiber_TFAE} For any geometric point $x\to X$ and 
	generization $t\to S$ of $f(x)$, the map 
	$\mathcal F_x=\Gamma(X_x,\mathcal F)\to\Gamma(X_x\times_{S_{f(x)}}S_t)$
	is an isomorphism, and this remains true after arbitrary base change in
	$S$.
	\item\label{4.4_rank1}
	After every base change along $\Spec V\to S$ with $V$ a rank 1 valuation
	ring with algebraically closed fraction field $K$, and for every
	geometric point $x\to X(:=X_V)$ in the special fiber, the map
	$\mathcal F_x
	=\Gamma(X_x,\mathcal F)\to\Gamma(X_x\times_{\Spec V}\Spec K,\mathcal F)$
	is an isomorphism.
	\end{enumerate}
	Moreover, the property of $\mathcal F$ being universally locally acyclic
	over $S$ is stable under base change in $S$ and arc-local on $S$.
\end{proposition-cite*}
\begin{proof}
That the property is stable under base change is immediate from the
stability under base change of \ref{4.4_ula_TFAE}.
That \ref{4.4_all_cover_TFAE} implies \ref{4.4_affine_TFAE} is clear.
As \ref{4.4_ula_TFAE}, \ref{4.4_milnor_fiber_TFAE} and \ref{4.4_rank1} are
manifestly \'etale-local on the source, that each is implied by
\ref{4.4_cover_TFAE} and each implies \ref{4.4_all_cover_TFAE} follows from
\cite[Theorem 4.4]{HS}. The same theorem also shows that
\ref{4.4_affine_TFAE} implies \ref{4.4_rank1} and that the property of
being arc-local on the base holds when $X$ and $S$ are qcqs and $f$ is
separated. This implies the conclusion in general, e.g. using
\ref{4.4_affine_TFAE}.
\comment{
Some more words on why \ref{4.4_cover_TFAE} implies \ref{4.4_ula_TFAE}
after arbitrary base change in $S$: let $S'\to S$ be any base change and
$x\to X\times_SS'$ be a geometric point. Then the image of $x$ in $X$ lies
in some $U_{ij}$, and in fact we can replace $X$ and $S$ by 
$U_{ij}\to V_j$. Shrinking $S'$, we may assume it too is affine.
Then $X\times_SS'$ is affine and $(X\times_SS',\mathcal F|_{X\times_SS'})$ 
is dualizable in $\mathcal C_{S'}$ by the stability of that property under
base change.

Some more words on arc-local on $S$ in general: let $S'\to S$ be an 
arc-cover such that $\mathcal F|_{S'}$ is universally locally acyclic over
$S'$. Fix $U$ and $V$ as in \ref{4.4_affine_TFAE}; we must show that 
$(U,\mathcal F|_U)$ is dualizable in $\mathcal C_V$. For this purpose we
may assume $V=S$. By \cite[Definition 1.1(3)]{BM} and stability under base
change we may assume $S'$ is affine. Then $U\times_SS'$ is affine, so
\ref{4.4_affine_TFAE} allows us to conclude that the inverse image of 
$\mathcal F$ there is a dualizable object of $\mathcal C_{S'}$, hence that
$\mathcal F|_U$ is a dualizable object of $\mathcal C_S=\mathcal C_V$, as
desired.}
\end{proof}
It's clear that Definition~\ref{sec:intro_ula} can just as well be made
for coefficient rings such as $\mathcal O_L/\ell^n$ with $L/\QQ_\ell$ a
finite extension (or more generally for discrete coefficient rings killed
by some power of $\ell$); i.e. for constructible \'etale sheaves (as
$D(?,\mathcal O_L/\ell^n)\simeq D_{\mathrm{cons}}(\et ?,\mathcal O_L/\ell^n)$ \cite[Proposition 7.1]{HRS}).
The previous proposition continues to hold in that setting, as does
the following result, which shows that this notion of universal local
acyclicity satisfies the desired formalism.
\begin{lemma*}\label{lem:stabilities}
Let $X$ and $Y$ be schemes locally of finite presentation over a scheme 
$S$.
\begin{enumerate}[label=(\roman*)]
\item\label{lem:stabilities:smooth} If $f:X\to Y$ is a smooth map over $S$ 
and $\mathcal F\in D(Y)$ is universally locally acyclic over $S$,
then $f^*\mathcal F$ is universally locally acyclic over $S$.
\item\label{lem:stabilities:proper} If $f:X\to Y$ is a proper 
finitely-presented map over $S$ and $\mathcal F\in D(X)$ is universally 
locally acyclic over $S$, then $f_*\mathcal F$ is universally locally
acyclic over $S$.
\item\label{lem:stabilities:base} If $S\to S'$ is a smooth map to a scheme 
$S'$ and $\mathcal F\in D(X)$ is universally locally acyclic over $S$, then 
$\mathcal F$ is universally locally acyclic over $S'$.
\item\label{lem:stabilities:support} If $i:X\hookrightarrow Y$ is a 
constructible closed immersion, then $\mathcal F\in D(X)$ is universally
locally acyclic over $S$ if and only if $i_*\mathcal F$ is universally
locally acyclic over $S$.
\end{enumerate}
\end{lemma*}
\begin{proof}
\ref{lem:stabilities:support}\quad
If $x\to X$ is a geometric point, then as $Y_x\times_YX=X_x$, 
$\mathcal F_x=\Gamma(X_x,\mathcal F)=\Gamma(Y_x,i_*\mathcal F)$ and
$\Gamma(X_x\times_{S_{f(x)}}t,\mathcal F)
=\Gamma(Y_x\times_{S_{f(x)}}t,i_*\mathcal F)$,
we may conclude that
$\mathcal F_x\to\Gamma(X_x\times_{S_{f(x)}}t,\mathcal F)$ is an
isomorphism if and only if
$\Gamma(Y_x,i_*\mathcal F)\to \Gamma(Y_x\times_{S_{f(x)}}t,i_*\mathcal F)$
is.

\ref{lem:stabilities:base} Let $g:S\to S'$ be our smooth morphism.
We may assume $X$, $S$ and $S'$ are qcqs and that the morphisms between 
them are separated. We may also replace $S'$ by $g(S)$ and assume $g$ is
fppf, hence a $v$-cover. As the property of being universally locally
acyclic is arc-local on the base, it will suffice to show that
$\mathcal F|_{X\times_{S'}S}$ is universally locally acyclic over $S$.
By hypothesis, $\mathcal F$ is universally locally acyclic over $S$.
Part \ref{lem:stabilities:smooth} now implies the desired conclusion.

For \ref{lem:stabilities:smooth} and \ref{lem:stabilities:proper}, one
reduces to the case of $X$ and $Y$ qcqs and separated over $S$ by
looking Zariski-locally using the previous proposition.
Therefore we may assume $X$, $Y$ and $S$ are all qcqs and that $X$ and $Y$
are separated over $S$. We also reduce to $\Lambda$ a finite extension $L$
of $\QQ_\ell$ or its ring of integers.
Let $\mathcal F_n:=\mathcal F/\ell^n$ if $\Lambda$ is integral.

\ref{lem:stabilities:smooth} Suppose $f$ has relative dimension $d$. Then
$f^*\mathbf D_{Y/S}(\mathcal F)=\mathbf D_{X/S}(f^*\mathcal F)(-d)[-2d]$
and
$(f\times_Sf)^*\sHom_{Y\times_SY}(\pi_1^*\mathcal F,\pi_2^!\mathcal F)
=\sHom_{X\times_SX}(\pi_1^*f^*\mathcal F,\pi_2^!f^*\mathcal F)(-d)[-2d]$,
where $\pi_2^!\mathcal F$ continues to denote $\lim_n\pi_2^!\mathcal F_n$
if $\Lambda$ is integral and $(\lim_n\pi_2^!\mathcal F_n)[\ell^{-1}]$ if
$\Lambda$ is rational.
In particular, both are constructible, and the morphism
\begin{equation*}\mathbf D_{X/S}(f^*\mathcal F)\boxtimes_S f^*\mathcal F\to
\sHom(\pi_1^*f^*\mathcal F,\pi_2^!f^*\mathcal F)\end{equation*}
is an isomorphism, so we can conclude by 
Corollary~\ref{HS:3.3_corollary} that $f^*\mathcal F$ is universally
locally acyclic over $S$.

\ref{lem:stabilities:proper} Local Verdier duality gives 
$f_*\DD_{X/S}(\mathcal F)=\DD_{Y/S}(f_*\mathcal F)$. 
In the case of integral coefficients, that and the K\"unneth formula allow
us to write
\begin{align*}
	\DD_{Y/S}(f_*\mathcal F)\boxtimes_S f_*\mathcal F
	&=(f\times_S f)_*(\DD_{X/S}(\mathcal F)\boxtimes_S\mathcal F) \\
	&\xra\sim(f\times_S f)_*\sHom(\pi_1^*\mathcal F,\pi_2^!\mathcal F) \\
	&=\lim_n(f\times_S\id)_*(\id\times_Sf)_*\sHom_{X\times_SX}((\id\times_Sf)^*\pi_1^*\mathcal F_n,\pi_2^!\mathcal F_n) \\
	&=\lim_n(f\times_S\id)_*\sHom_{X\times_SY}(\pi_1^*\mathcal F_n,\pi_2^!f_*\mathcal F_n) \\
	&=\lim_n(f\times_S\id)_*\sHom_{X\times_SY}(\pi_1^*\mathcal F_n,(f\times_S\id)^!\pi_2^!f_*\mathcal F_n) \\
	&=\sHom_{Y\times_SY}(\pi_1^*f_*\mathcal F,\pi_2^!f_*\mathcal F),
\end{align*}
where here we use $\pi_1$ for both projections in the diagram below
\begin{equation*}\begin{tikzcd}[column sep=10]
	X\times_S X\arrow[rr,"\id\times_Sf"]
	\arrow[dr,"\pi_1"']&&X\times_SY\arrow[dl,"\pi_1"] \\ &X,
\end{tikzcd}\end{equation*}
and likewise for $\pi_2$. Of course, the same argument works for rational 
coefficients after choosing an integral model $\mathcal F_0$ and pulling 
out both the limit and the colimit in
$\pi_2^!\mathcal F=\colim_\ell\lim_n\pi_2^!(\mathcal F_0/\ell^n)$, allowing 
us to conclude by Corollary~\ref{HS:3.3_corollary} that $f_*\mathcal F$ is
universally locally acyclic over $S$.

This completes the proof of the lemma.
\comment{
More words for \ref{lem:stabilities:proper}: recall that a 
finitely-presented map is qcqs by definition. Now cover $S$ by affines
$V_i$, cover $Y$ by affines $U_{ij}$ as in Proposition~\ref{prop:HS4.4}.
Then $f^{-1}(U_{ij})$ is quasi-compact as $f$ is quasi-compact, and it's
quasi-separated since it's quasi-separated over $U_{ij}$ (as the base
change of $f$ along $U_{ij}\hookrightarrow Y$).}
\end{proof}

\subsection{} We can use the results of 
\S\ref{sec:C_S_internal_homs}--\ref{sec:converse_internal_hom}
to obtain some basic facts about universal local acyclicity.
\begin{lemma*}\label{lem:id_lisse}\begin{enumerate}[label=(\roman*)]
\item\label{lem:id_lisse:id} Let $X$ be a scheme and $\mathcal F\in D(X)$.
Then $\mathcal F$ is universally locally acyclic over $X$ if and only if 
$\mathcal F$ is lisse.
\item\label{lem:id_lisse:smooth} Let $f:X\to S$ be a smooth map of schemes 
and $\mathcal F\in D(X)$ lisse. Then $f$ is universally locally acyclic
relative to $\mathcal F$.
\end{enumerate}
\end{lemma*}
\begin{proof}\ref{lem:id_lisse:id}
As the properties of being universally locally acyclic and of being lisse
are both Zariski-local on the source \cite[Lemma 4.5]{HRS}, we may assume
$X$ is qcqs. The property of $\mathcal F$ being universally locally acyclic
over $X$ implies that for any $\mathcal G$ in $D(X)$,
$\sHom(\mathcal F,\Lambda)\otimes\mathcal G
\to\sHom(\mathcal F,\mathcal G)$
is an isomorphism by 
Proposition~\ref{sec:converse_internal_hom},\footnote{We may approximate
$\mathcal F$ and $\mathcal G$ by sheaves $\mathcal F'$ and $\mathcal G'$
over a finite extension of $\QQ_\ell$. Here, $\mathcal F'$ must be
universally locally acyclic, as this can be verified by the vanishing of
certain cones, and extension of scalars from $\mathcal O_L$ to 
$\mathcal O_E$ or from $L$ to $E$ ($L/\QQ_\ell$ finite, $E/\QQ_\ell$
algebraic, $L\subset E$) is faithfully flat.} which 
implies that $\mathcal F$ is dualizable in $\mathcal D(X)$
\cite[Lemma 4.6.1.6]{HA}, and conversely for $\mathcal G=\mathcal F$ by
Corollary~\ref{HS:3.3_corollary}.\footnote{Indeed, if 
$\mathcal F$ is lisse, then 
$\mathbf D_{X/X}(\mathcal F)=\sHom(\mathcal F,\Lambda)=\mathcal F^\vee$ is,
too (in particular constructible).}

\ref{lem:id_lisse:smooth} It follows from \ref{lem:id_lisse:id} in view of
Lemma~\ref{lem:stabilities}\ref{lem:stabilities:base}.
\end{proof}	
Of course, the same is true for constructible \'etale sheaves of
$\ZZ/\ell^n$-modules using \cite[Propositions 3.3 \& 3.4(iv)]{HS}.

\subsection{} In this section and the next we study universal local
acyclicity over regular 0- and 1-dimensional bases and recover some
familiar results.
\begin{lemma*}\label{lem:ula_field}
	If $X$ is a variety, then any $\mathcal F\in D(X)$ is universally
	locally acyclic over $\Spec k$.
\end{lemma*}
\begin{proof}
	As universal local acyclicity is $v$-local on the base, we may assume
	$k$ is algebraically closed.
	The conclusion is now immediate from \cite[Theorem 4.1]{HS}.
\end{proof}

\subsection{}
For \'etale sheaves, local acyclicity over a smooth curve is
tantamount to all the vanishing cycles being null. The same is true with
$\ell$-adic coefficients.
\begin{lemma*}\label{lem:ula_curve}
	Let $f:X\to Y$ be a map locally of finite type with $Y$ a regular
	1-dimensional scheme, and $\mathcal F\in D(X)$. Then $f$ is universally 
	locally acyclic rel. $\mathcal F$ if and only if all the vanishing 
	cycles are zero; i.e. if and only if $f$ is universally locally acyclic
	at every point of $|X|$ rel.\,$\mathcal F$.
\end{lemma*}
\begin{proof}
The criterion of Theorem~\ref{prop:HS4.4}\ref{4.4_rank1} allows us to
reduce to $Y=\Spec\mathcal O_{Y,y}$. Indeed, if $g:\Spec W\to Y$ is any map
from the spectrum of a rank 1 absolutely integrally closed valuation ring,
$g$ either factors via a point of $Y$ or via a surjective map to 
$\Spec\mathcal O_{Y,y}$ for some $y\in Y$ of codimension 1.
As universal local acyclicity over a field is automatic by
Lemma~\ref{lem:ula_field}, we may assume we are in the latter case.

So let $Y=\Spec\mathcal O_{Y,y}$. Let $V_y$ denote the normalization of the 
strict henselization of $\mathcal O_{Y,y}$ (with respect to some choice of
geometric point centered on $y$) in $\alg{(\Frac Y)}$;
$V_y$ is an absolutely integrally closed valuation ring. 
The maps $\mathcal O_{Y,y}\to W$ and $\mathcal O_{Y,y}\to V_y$
are $v$-covers, so since universal local acyclicity can be checked
$v$-locally, $f_W:X\times_Y\Spec W\to\Spec W$ is universally locally 
acyclic rel. $\mathcal F$ if and only if $f_{V_y}$ is, and \cite[Theorem~4.1]{HS} 
tells us this is tantamount to the vanishing cycles (computed relative to
$V_y$) being zero.
\end{proof}

\subsection{} Let $E/\QQ_\ell$ be a finite extension.
As reduction modulo $\ell$ induces a symmetric monoidal functor
$\mathcal C_{S,\mathcal O_E}\to\mathcal C_{S,\mathcal O_E/\ell}$ in the
setting of \S\ref{sec:C_S}, it's clear that if $f:X\to S$ is a locally
finitely-presented map of schemes and $\mathcal F\in D(X,\mathcal O_E)$
is universally locally acyclic over $S$, then $\mathcal F/\ell$ is, too.
The converse holds as well.
\begin{lemma*}\label{lem:ula_mod_ell}
If $f:X\to S$ is a locally finitely-presented map of schemes and 
$\mathcal F\in D(X,\Lambda)$ with $\Lambda=\mathcal O_E$, 
$E/\QQ_\ell$ finite, then $f$ is universally locally acyclic
rel.\,$\mathcal F$ if and only if it is rel.\,$\mathcal F/\ell$.
\end{lemma*}
\begin{proof}
	We may assume $X$ is affine and finitely-presented over $S$, also 
	affine, at which point the forward direction is a consequence of the 
	fact that reduction modulo $\ell$ induces a symmetric monoidal functor
	$\mathcal C_{S,\Lambda}\to\mathcal C_{S,\Lambda/\ell}$.
	For the converse, we may assume by the criterion of
	Proposition~\ref{prop:HS4.4}\ref{4.4_rank1} that $S$ is a rank~1
	absolutely integrally closed valuation ring, and it suffices to
	know that $\mathcal F\to j_*j^*\mathcal F$ is an isomorphism, where $j$
	is the (base extension of the) immersion of the generic point of $S$.
	\comment{We may approximate $\mathcal F$ by a sheaf
	$\mathcal F'\in D(X,\mathcal O_E)$ with $E/\QQ_\ell$ finite, and the
	vanishing of the cone of the map
	$\mathcal F/\ell\to j_*j^*(\mathcal F/\ell)$ implies the same with
	$\mathcal F$ replaced by $\mathcal F'$.\footnote{Say $\mathcal F$ is an
	$\mathcal O_L$-module with $L/\QQ_\ell$ infinite. As the cone
	$\mathcal G:=\cofib(\mathcal F'/\ell\to j_*j^*(\mathcal F'/\ell))$
	is an \'etale sheaf, it will have a nonzero stalk if it's nonzero, say
	at a geometric point $x\to X$. So $\Gamma(X_x,\mathcal G)\ne0$.
	Then $\Gamma(X_x,\mathcal G\otimes_{\mathcal O_E}\mathcal O_L)
	=\Gamma(X_x,\mathcal G)\otimes_{\mathcal O_E}\mathcal O_L\ne0$ as
	$\mathcal O_E\subset\mathcal O_L$ is faithfully flat, so
	$0\ne\mathcal G\otimes_{\mathcal O_E}\mathcal O_L
	=\cofib(\mathcal F/\ell\to j_*j^*(\mathcal F/\ell))$.}}
	Here, $j_*j^*\mathcal F$ is constructible by \cite[Theorem 4.1]{HS},
	\comment{So we may replace $\mathcal F$ by $\mathcal F'$ at which point}
	so both $\mathcal F$ and $j_*j^*\mathcal F$ are derived
	$\ell$-complete. It therefore suffices by Lemma~\ref{sec:cons_facts} to
	know that
	$\mathcal F/\ell\to(j_*j^*\mathcal F)\otimes_{\ZZ_\ell}\ZZ_\ell/\ell
	=j_*j^*(\mathcal F/\ell)$ is an isomorphism (the latter isomorphism as
	$\ZZ_\ell/\ell$ is perfect). This is guaranteed by the hypothesis that
	$f$ is universally locally acyclic rel.\,$\mathcal F/\ell$.
\end{proof}

\subsection{}
Similarly, universal local acyclicity can be checked after extending 
scalars.
\begin{lemma*}\label{lem:integral_extension_ula}
	Let $f:X\to S$ be a locally finitely-presented morphism of schemes.
	Let $\QQ_\ell\subset E\subset L$ with $E/\QQ_\ell$ finite and $L/E$ 
	algebraic, and let $\mathcal G'\in D(X,\mathcal O_E)$.
	Then $\mathcal G'$ is universally locally acyclic over $S$ if and only
	if $\mathcal G:=\mathcal G'\otimes_{\mathcal O_E}\mathcal O_L$ is.
\end{lemma*}
\begin{proof}
We may assume $X$ and $S$ are qcqs and that $f$ is separated.
As extension of scalars from $\mathcal O_E$ to $\mathcal O_L$ induces
a symmetric monoidal functor
$\mathcal C_{S,\mathcal O_E}\to\mathcal C_{S,\mathcal O_L}$, the forward
direction is obvious. For the converse, suppose $\mathcal G$ is universally
locally acyclic over $S$. We may suppose by
Proposition~\ref{prop:HS4.4}\ref{4.4_rank1} that $S$ is a rank 1 absolutely
integrally closed valuation ring and that $\mathcal G\to j_*j^*\mathcal G$
is an isomorphism, where $j$ is the (base change of the) immersion of the
generic point of $S$, and we must show that the same is true for
$\mathcal G'$.
As the cone $\mathcal K':=\cofib(\mathcal G'\to j_*j^*\mathcal G')$ is
a constructible $\mathcal O_E$-sheaf by \cite[Theorem 4.1]{HS}, 
Lemma~\ref{lem:stalks} guarantees it will have a nonzero stalk if it's
nonzero, say at a geometric point $x\to X$.
As $\mathcal O_E\subset\mathcal O_L$ is faithfully flat and
$\mathcal K_x
=\Gamma(X_x,\mathcal K:=\mathcal K'\otimes_{\mathcal O_E}\mathcal O_L)
=\Gamma(X_x,\mathcal K')\otimes_{\mathcal O_E}\mathcal O_L
=\mathcal K'_x\otimes_{\mathcal O_E}\mathcal O_L$,
$\mathcal K'_x=0$ if and only if $\mathcal K_x=0$.
As $\mathcal K=\cofib(\mathcal G\to j_*j^*\mathcal G)=0$ by assumption,
we may conclude.
\end{proof}

\subsection{}
It follows from the definition of duals in the category $\mathcal C_S$
and Proposition~\ref{prop:duals_in_C_S}
that the property of being universally locally acyclic often commutes with
Verdier duality when the base is a smooth variety.
\begin{lemma*}\label{lem:verdier_duality}
	Let $f:X\to S$ be a map of (not necessarily separated) varieties with
	$S$ smooth, and assume $f$ separated or $X$ smooth.
	Let $\mathcal F\in D(X)$.
	Then $f$ is universally locally acyclic rel.\,$\mathcal F$ if and only
	if it is rel.\,$D\mathcal F$, where $D$ denotes Verdier duality.
\end{lemma*}
\begin{proof}
	We may assume $X$ and $S$ connected of dimension $r$ and $d$, 
	respectively.
	Let $a,b:S,X\to\Spec k$ be the structural maps. Then $a^!\Lambda$ is
	defined, equals $\Lambda(d)[2d]$ and is a dualizing complex on $S$.
	Therefore if $f$ is separated,
	$D\mathcal F:=\sHom(\mathcal F,b^!\Lambda=f^!a^!\Lambda)$ is given by 
	$\sHom(\mathcal F,f^!\Lambda)(d)[2d]
	=\mathbf D_{X/S}(\mathcal F)(d)[2d]$,
	which is universally locally acyclic over $S$ if and only if 
	$\mathcal F$ is.\footnote{With these hypotheses, $f^!\Lambda$ and
	$D\mathcal F$ are constructible \cite[Lemmas 6.7.13 \& 6.7.19]{BS}.}
	
	If $X$ is smooth, then $b^!\Lambda$ is defined, equals $\Lambda(r)[2r]$
	and is a dualizing complex on $X$.
	For the purposes of checking universal local acyclicity, we may assume
	$X$ and $S$ are affine and then run the previous argument.
	\comment{(The check that the biduality morphism is an isomorphism is 
	Zariski-local, which is the same as the check that 
	$b^!\Lambda=f^!a^!\Lambda$ is a dualizing complex on $X$. The point is
	that we know that it is Zariski-locally, and that's enough.
	To say that $a^!\Lambda$ or $b^!\Lambda$ `are defined' doesn't mean
	that I claim that $a_!$ or $b_!$ are. We know that $a^!$ and $b^!$
	are defined Zariski-locally on $S$ or $X$, respectively, and if $S$ or
	$X$ is smooth then we know that these sheaves glue to a global sheaf on
	$S$ or $X$ which is therefore a dualizing complex.
	The isomorphism $\sHom(\mathcal F,f^!\Lambda)(d)[2d]
	=\mathbf D_{X/S}(\mathcal F)(d)[2d]$ can be checked on stalks using
	Lemma~\ref{sec:cons_facts}. The definition of universal local acyclicity
	shows that $\mathbf D_{X/S}(\mathcal F)(d)[2d]$ is universally
	locally acyclic if and only if $\mathbf D_{X/S}(\mathcal F)$ is, and then we use
	the biduality isomorphism
	$D(\mathbf D_{X/S}(\mathcal F)(d)[2d])=DD\mathcal F=\mathcal F$.)}
\end{proof}

\subsection{} The following result was shown by Illusie 
\cite[Proposition 2.10]{Illusie} with torsion coefficients under the
hypothesis of strong local acyclicity and with $g$ an open immersion. 
It holds most generally for \'etale sheaves, so let $\Lambda'$ denote some
discrete ring killed by a power of $\ell$ and denote by 
$\mathcal D(\et X,\Lambda')$ the left-completion
of the derived $\infty$-category of $\Lambda'$-modules on a scheme $X$ with
homotopy category $D(\et X,\Lambda')$. Let 
$\mathcal D_{\mathrm{cons}}(\et X,\Lambda')
\subset\mathcal D(\et X,\Lambda')$
denote the full subcategory on those objects which are Zariski-locally
lisse (i.e. dualizable in $\mathcal D(\et X,\Lambda')$) along a finite
subdivision into constructible locally closed subschemes
(see \cite[\S7.1]{HRS}).
\begin{proposition*}\label{prop:illusie}
	Let $f,g:X,Y\to S$ be morphisms locally of finite presentation between
	schemes $X$, $Y$ and $S$ and let 
	$\mathcal F\in D_{\mathrm{cons}}(\et X,\Lambda')$ 
	and $\mathcal G\in D(\et Y,\Lambda')$. Suppose $f$ is universally
	locally acyclic relative to $\mathcal F$. Then the natural map
	$\mathcal F\otimes f^*g_*\mathcal G\to 
	g_*(g^*\mathcal F\otimes f^*\mathcal G)$
	is an isomorphism, where $f$ and $g$ continue to denote their base
	extensions $X\times_SY\to Y,X$.
	
	The same holds for $\ell$-adic sheaves $\mathcal F\in D(X)$ and
	$\mathcal G\in D(Y)$ if $Y$ and $S$ are supposed qcqs and either $S$
	noetherian quasi-excellent or $g$ proper.
\end{proposition*}
The latter hypotheses ensure that $g_*$ preserves constructibility
\cite[Lemma 6.7.2]{BS}.
Note that if $X$ and $S$ are qcqs and $f$ is separated, then an \'etale
sheaf which is $f$-universally locally acyclic is necessarily already
constructible \cite[Proposition 3.4(iii)]{HS}.
Note also that the proposition implies the version with $D(\et Y,\Lambda')$
replaced by its ordinary non-left-completed version provided that 
$\mathcal G$ is bounded below (i.e. in $D^+$), and that left-completion is
superfluous if $\et X$ has locally finite $\ell$-cohomological dimension by 
\cite[Proposition 3.3.7(2)]{BS}.
\begin{proof}
	Let's first consider the case of \'etale sheaves.
	We may assume $X$ and $S$ are affine.
	The object $(X,\mathcal F)$ is then dualizable in the version of the 
	2-category $\mathcal C_S$ where $D(?)$ is replaced by 
	$D(\et ?,\Lambda')$; i.e. setting (A) of \cite{HS}. We write
	\begin{align*}
		\mathcal F\otimes f^*g_*\mathcal G
		&=\mathcal F\boxtimes_S g_*\mathcal G
		=\mathbf D_{X/S}\mathbf D_{X/S}\mathcal F\boxtimes_Sg_*\mathcal G \\
		&=\sHom_X(\mathbf D_{X/S}\mathcal F,f^!g_*\mathcal G)
		=g_*\sHom_{X\times_SY}(g^*\mathbf D_{X/S}\mathcal F,f^!\mathcal G) \\
		&=g_*(\mathbf D_{X/S}\mathbf D_{X/S}\mathcal F\boxtimes_S\mathcal G)
		=g_*(g^*\mathcal F\otimes f^*\mathcal G).
	\end{align*}
	That this morphism coincides with the map in the proposition follows
	from the fact that the adjoints of both maps are induced by the counit 
	$g^*g_*\to\id$.
	Here, the third isomorphism follows from \cite[Proposition 3.4(iv)]{HS}.
	We may check the isomorphism
	$\mathbf D_{X/S}\mathbf D_{X/S}\mathcal F\boxtimes_S\mathcal G
	\to\sHom(g^*\mathbf D_{X/S}\mathcal F,f^!\mathcal G)$
	locally on $X\times_SY$ and assume that $Y$
	is also affine, at which point it follows from \emph{loc. cit.}
	
	The case of $\ell$-adic sheaves immediately reduces to the case of a
	finite extension of $\QQ_\ell$. Then, with integral coefficients, it 
	follows directly from the torsion \'etale case by reducing mod $\ell$
	and using Lemma~\ref{lem:reduction_conservative}.
	The case of rational coefficients can be handled by the same argument as
	in the torsion \'etale case, with the reference to
	\cite[Proposition 3.4(iv)]{HS} being replaced by one to 
	Proposition~\ref{prop:HS_3.3_dualizable_internal_hom}
	(so that $f^!\mathcal G$ means 
	$(\lim_nf^!(\mathcal G_0/\ell^n))[\ell^{-1}]$,
	where $\mathcal G_0$ is some integral model for $\mathcal G$).
\end{proof}
\begin{remark*}
	One recovers smooth base change as a special case of the previous
	proposition. Indeed, suppose $f$ is smooth and set $\mathcal F=\Lambda$.
	Then $f$ is universally locally acyclic rel.\,$\mathcal F$ 
	by Lemma~\ref{lem:id_lisse}\ref{lem:id_lisse:smooth}, and
	the proposition says that $f^*g_*\mathcal G\to g_*f^*\mathcal G$ is an
	isomorphism.
\end{remark*}

\section{Singular support}\label{sec:SS}
In this section, we construct the singular support of an $\ell$-adic sheaf
on a smooth variety and prove Theorem~\ref{th:SS}. We start with the case
of an integral sheaf, as it quickly reduces to the torsion \'etale case.
We use without further ado the definitions of
\S\ref{sec:transversality_def} \& \S\ref{sec:ss_def}.

\subsection{}\label{sec:integral_SS}
Suppose $X$ is a smooth variety, $\Lambda=\mathcal O_E$ with
$E/\QQ_\ell$ a finite extension, and $\mathcal F\in D(X,\Lambda)$.
Then $\mathcal F/\ell$ is an \'etale $\ZZ/\ell$-sheaf, and has a notion of
singular support.
It follows from Lemma~\ref{lem:ula_mod_ell} that a test pair is
$\mathcal F$-acyclic if and only if it is $\mathcal F/\ell$-acyclic.
Therefore $SS(\mathcal F)$ exists and equals $SS(\mathcal F/\ell)$.
For the same reason (or by Lemma~\ref{lem:vanishing_cycles_mod_ell}), we
have equality of the weak singular supports.
The equality $SS(\mathcal F)=SS(D\mathcal F)$ follows from
Lemma~\ref{lem:verdier_duality} in view of the equality 
$SS(\mathcal F)=SS^w(\mathcal F)$.\footnote{The corresponding facts for
\'etale sheaves appear as \cite[Corollary 4.5(4) \& Corollary 4.9]{Saito}.}
The rest of Theorem~\ref{th:SS} for 
$\mathcal F$ follows from the corresponding facts for $\mathcal F/\ell$.

Now suppose $\Lambda=\mathcal O_L$ with $L$ an infinite algebraic extension
of $\QQ_\ell$. Then we can approximate $\mathcal F$ by some
$\mathcal F'\in D(X,\mathcal O_E)$, $E/\QQ_\ell$ finite, and
$\mathcal F'$ has singular support satisfying Theorem~\ref{th:SS}.
We are done by Lemma~\ref{lem:integral_extension_ula}.

This completes the proof of Theorem~\ref{th:SS} for integral $\Lambda$,
modulo the first sentence of Theorem~\ref{th:SS}\ref{th:SS:constituents}
which will be addressed in \S\ref{sec:SSw} below.

\subsection{}
We turn now to the case of rational $\Lambda$.
Here, picking an integral model and computing its singular 
support will generally give the wrong answer, as is seen by taking the
direct sum of a torsion-free $\ZZ_\ell$-local system with a torsion \'etale
constructible sheaf which is not lisse: the rationalization should have
singular support given by the zero section only, but the singular support
of such an integral model will have other irreducible components.
To remedy this, we analyze Beilinson's constructions in \cite{Sasha},
elaborating on some of them, and explain why they continue to work with
rational $\ell$-adic coefficients.

Let $X$ be a smooth variety and $\mathcal F$ in $D(X)$.
Let $\mathcal C(\mathcal F)$ denote the set of those conical closed subsets 
$C$ of $T^*X$ so that $\mathcal F$ is micro-supported on $C$.
The proof that $T^*X\in\mathcal C(\mathcal F)$ is unchanged from that of
\cite[Lemma 1.3]{Sasha}, and uses Lemma~\ref{lem:ula_field}, the
stability of universal local acyclicity under base change, and
Lemma~\ref{lem:stabilities}\ref{lem:stabilities:smooth}.

\subsubsection*{Legendre \& Radon transforms}
We now recall the definitions of the Legendre and Radon transforms,
referring the reader to \cite[\S1.6]{Sasha} for more detail.
Let $V$ be a vector space of dimension $n+1$, $V^\vee$ its dual,
$\PP$ and $\PP^\vee$ the corresponding projective spaces,
$Q\hookrightarrow\PP\times\PP^\vee$ the incidence correspondence,
$p,p^\vee:Q\to\PP,\PP^\vee$ the projections.

The \emph{Legendre transform} is a pair of identifications
$P(T^*\PP)\xleftarrow\sim Q\xrightarrow\sim P(T^*\PP^\vee)$.
Given a point\footnote{By abuse of notation, $x\in\PP$ will often mean
$x=\Spec k'\to\PP$, where $k'$ is a finite extension of $k$.}
$x\in\PP$ and a line (through the 
origin) $L\subset T^*_x\PP$, we obtain first a line through $x$ in $\PP$
and then a hyperplane in $\PP$ containing $x$ given as the perpendicular to
that line. This hyperplane determines a point $x^\vee\in\PP^\vee$. The 
point $x$ determines a hyperplane $x^\perp$ in $\PP^\vee$, whose 
conormal gives rise to a line $L^\vee$ in $T^*_{x^\vee}\PP^\vee$.
The Legendre transform sends $(x,L)\mapsto(x^\vee,L^\vee)$ and running the
same algorithm again on $(x^\vee,L^\vee)$ gives its inverse.

The \emph{Radon transform} functors $R,R^\vee$ are defined as
$R:=p^\vee_*p^*[n-1]:D(\PP)\to D(\PP^\vee)$ and
$R^\vee:=p_*p^{\vee*}[n-1]:D(\PP^\vee)\to D(\PP)$.
The functors $R$ and $R^\vee(n-1)$ are both left and right adjoint.
Brylinski showed \cite[Corollaire 9.15]{Brylinski} that if $\mathcal M$ is
a perverse sheaf on $\PP$, then ${^p\mathcal H^i}(R(\mathcal M))$ is 
lisse for $i\ne0$, and that the functor 
$?\mapsto{^p\mathcal H^0}R(?)$ induces an equivalence
between the quotient of $\Perv(\PP)$ by the (full) subcategory of lisse 
perverse sheaves\footnote{Let $a:\PP\to\Spec k$ be the structure morphism.
A perverse sheaf on $\PP$ is lisse if and only if it is the $a^*[n]$
of a local system (i.e. lisse perverse sheaf) on $\Spec k$, as can be seen 
from Corollary~\ref{cor:lisse_t} and the argument of the proof of
Proposition~\ref{prop:lisse_equiv}.} and the similar quotient of
$\Perv(\PP^\vee)$. This implies in particular that the cones of the
adjunction morphisms $RR^\vee(n-1)\to\id$ and $\id\to R^\vee R(n-1)$ 
(on all of $D(\PP)$ and $D(\PP^\vee)$) are lisse, that if $\mathcal M$ is
lisse then $R(\mathcal M)$ is lisse, and that if $\mathcal M$ is
irreducible and not lisse, then ${^p\mathcal H^0}R(\mathcal M)$ has a
single non-lisse irreducible perverse sheaf in its Jordan-H\"older series.

\subsection{}
Let $X$ be a smooth variety, $\mathcal F\in D(X)$ and $C$ a closed conical
subset of $T^*X$.
\begin{lemma-cite*}[{\cite[Lemma 2.1]{Sasha}}]\label{lem:B2.1}
\begin{enumerate}[label=(\roman*)]
\item\label{lem:B2.1:base} The base of $SS^w(\mathcal F)$ equals the 
support of $\mathcal F$.
\item The sheaf $\mathcal F$ is micro-supported on the zero section 
$T_X^*X$ if and only if $\mathcal F$ is lisse.
\item\label{lem:B2.1:thick} The sheaves that are micro-supported on $C$ 
form a thick subcategory of $D(X)$.
\end{enumerate}\end{lemma-cite*}
The original proof goes through unchanged;
the fact that $(\id_X,\id_X)$ is a $\mathcal F$-acyclic test pair if and
only if $\mathcal F$ is lisse translates to the statement of
Lemma~\ref{lem:id_lisse}\ref{lem:id_lisse:id}, while the fact that
$\mathcal F$ is micro-supported on $T^*_XX$ if it is lisse follows from
Lemma~\ref{lem:id_lisse}\ref{lem:id_lisse:smooth}.

\subsection{}
If $r:X\to Z$ is a map of smooth varieties and $C$ is a conical closed
subset of $T^*X$ whose base is proper over $Z$, $r_\circ C$ is defined as
the image of $(dr)^{-1}(C)\subset T^*Z\times_ZX$ by the projection
$T^*Z\times_ZX\to T^*Z$. It is a conical closed subset of $T^*Z$.
\begin{lemma-cite*}[{\cite[Lemma 2.2]{Sasha}}]
\begin{enumerate}[label=(\roman*)]
\item If $Y$ is smooth, $f:Y\to X$ is $C$-transversal and 
$C\in\mathcal C(\mathcal F)$, then 
$f^\circ C\in\mathcal C(f^*\mathcal F)$.
\item\label{lem:B2.2:ii}
If $r:X\to Z$ with $Z$ smooth and $C\in\mathcal C(\mathcal F)$ with
base proper over $Z$, then $r_\circ C\in\mathcal C(r_*\mathcal F)$.
\end{enumerate}
\end{lemma-cite*}
The original proofs go through unchanged using
Lemma~\ref{lem:stabilities}\ref{lem:stabilities:proper} \&
\ref{lem:stabilities:support} for \ref{lem:B2.2:ii}.
Same for \cite[Lemmas 2.3, 2.4 \& 2.5]{Sasha}; the proof of part (i) of the
latter uses Lemma~\ref{lem:stabilities}\ref{lem:stabilities:support}
(or just the compatibility of the formation of vanishing cycles with proper 
base change).

\subsection{} In this section we explain how to see in our setting the
existence of the singular support and the upper dimension estimate on its
dimension. If $g:Y\to Z$ is a map of varieties and $\mathcal G\in D(Y)$,
then let $E_g(\mathcal G)$ denote the smallest closed subset of $Y$ such
that $g$ is universally locally acyclic rel.\,$\mathcal G$ on 
$Y\smallsetminus E_g(\mathcal F)$.
\begin{proposition-cite*}[{\cite[Theorems 3.1 \& 3.2]{Sasha}}]
\begin{enumerate}[label=(\roman*)]
\item\label{prop:SS_exists:X} Every sheaf $\mathcal F\in D(X)$ on a smooth
variety $X$ has singular support. One has $\dim SS(\mathcal F)\leq\dim X$.
\item\label{prop:SS_exists:P} If $X=\PP$, the projectivization
$P(SS(\mathcal F))\subset P(T^*\PP)$ equals the Legendre transform of
$E_p(p^{\vee*}R(\mathcal F))\subset Q$.
If $\mathcal F$ vanishes at the generic point of $\PP$ then 
$SS(\mathcal F)$ is the cone over $P(SS(\mathcal F))$, otherwise 
$SS(\mathcal F)$ is the union of this cone and $T^*_XX$.
\end{enumerate}
\end{proposition-cite*}
The existence statement of~\ref{prop:SS_exists:X} follows
from~\ref{prop:SS_exists:P}.
The proof of~\ref{prop:SS_exists:P} goes like the proof of
\cite[Theorem 3.2]{Sasha}, with
Lemma~\ref{lem:stabilities}\ref{lem:stabilities:proper} in lieu of 3.13
(i.e. Lemma~3.9) and Lemma~\ref{lem:stabilities}\ref{lem:stabilities:base} 
in lieu of \cite[Lemme 2.14]{D}.
This proves Theorem~\ref{th:SS}\ref{th:SS:existence}.

Now suppose $\mathcal F$ is a sheaf on $\PP$, $\mathcal F_0$ is any
integral model for $\mathcal F$ and $\mathcal F_0'$ is any approximation of
$\mathcal F_0$ over the ring of integers of a finite extension of 
$\QQ_\ell$.
Then $E_p(p^{\vee*}R(\mathcal F))\subset E_p(p^{\vee*}R(\mathcal F_0))
=E_p(p^{\vee*}R(\mathcal F_0'))=E_p(p^{\vee*}R(\mathcal F_0'/\ell))$.
Here, the first equality is by Lemma~\ref{lem:integral_extension_ula}, the
second by Lemma~\ref{lem:ula_mod_ell}, and the inclusion is obvious.
Beilinson shows \cite[Theorem 3.6]{Sasha} that for any constructible
\'etale $\ZZ/\ell$-sheaf $\mathcal G$ on $\PP^\vee$, one has
$\dim E_p(p^{\vee*}\mathcal G)\leq n-1$.
Therefore $\dim E_p(p^{\vee*}R(\mathcal F))\leq n-1$ also, which proves the
dimension estimate in part \ref{prop:SS_exists:X} and also establishes
Theorem~\ref{th:SS}\ref{th:SS:integral_comparison}, as 
$P(SS(\mathcal F_0))$ coincides with the Legendre transform of 
$E_p(p^{\vee*}R(\mathcal F_0))$.

\begin{remark*}\label{remark:SS_finite_extn}
As in the case of \'etale sheaves, this description of the singular support
implies its invariance under finite extensions of the base field: i.e. let
$k'/k$ be a finite extension; then 
$SS(\mathcal F_{k'})=SS(\mathcal F)_{k'}$ since 
$E_{p_{k'}}(p_{k'}^{\vee*}R(\mathcal F_{k'}))=E_p(p^{\vee*}R(\mathcal F))_{k'}$.\footnote{Let $\mathcal G:=p^{\vee*}R(\mathcal F)$. Then we
always have $E_{p_{k'}}(\mathcal G_{k'})\subset E_p(\mathcal G)_{k'}$.
Suppose the inclusion were strict. Then the image $U$ in $Q$ of
$Q_{k'}\smallsetminus E_{p_{k'}}(\mathcal G_{k'})$ would be an open set
strictly larger than $Q\smallsetminus E_p(\mathcal G)$. I claim that
$\mathcal G$ is $p$-universally locally acyclic on $U$.
Indeed, if $x\to U$ is a geometric point, $(U_x)_{k'}$ is the
disjoint union of the strict localizations of $Q_{k'}$ at the finitely many
points $y_i\in x\otimes_kk'$, and likewise for $(\PP_{p(x)})_{k'}$.
As $\Gal(k'/k)$ acts transitively on the $y_i$ (we may assume $k'/k$ is
Galois) and our sheaf $\mathcal G$ began life on $Q$,
$\mathcal G_{y_i}\to\Gamma((U_{k'})_{y_i}\times_{(\PP_{k'})_{p(y_i)}}t,\mathcal G)$
is an isomorphism for all $i$ because it is an isomorphism for one $i$ by
hypothesis. The same continues to hold if we replace $\PP$ by any 
$S\to\PP$.}
\end{remark*}

\subsection{}
To prove the rest of Theorem~\ref{th:SS}, at this point it is necessary to
state Beilinson's theorem on the singular support and the ramification
divisor, which continues to hold in our setting.

Let $i:\PP\hookrightarrow\tilde\PP$ denote the Veronese embedding of degree
$d>1$. Let $\tilde Q\subset\tilde\PP\times\tilde\PP^\vee$ denote the
incidence correspondence, and let 
$\tilde p,\tilde p^\vee:\tilde Q\to\tilde\PP,\tilde\PP^\vee$
denote the projections. Let $\tilde R,\tilde R^\vee$ denote the Radon
transforms with respect to $\tilde p$ and $\tilde p^\vee$.

A closed conical subset $C\subset T^*\PP$ gives rise to one in
$T^*\tilde\PP$ given by $i_\circ C$. Via the Legendre transform, its
projectivization $P(i_\circ C)$ comes with a map to $\tilde\PP^\vee$.
Let $D_C$ denote the image of $P(i_\circ C)$ in $\tilde\PP^\vee$.

If $\mathcal F$ is a sheaf on $\PP$, then let $D_{\mathcal F}$ denote the
smallest closed subset of $\tilde\PP^\vee$ such that
$\tilde R(i_*\mathcal F)$ is lisse on 
$\tilde\PP^\vee\smallsetminus D_{\mathcal F}$.
\begin{proposition-cite*}[{\cite[Theorem 1.7]{Sasha}}]\label{prop:B1.7}
	The closed subset $D_{\mathcal F}\subset\tilde\PP^\vee$ is a divisor.
	For each irreducible component $D_\gamma$ of $D_{\mathcal F}$, there is
	a unique irreducible closed conical subset $C_\gamma\subset T^*\PP$ of
	dimension $n$ with $D_\gamma=D_{C_\gamma}$.
	One has $SS(\mathcal F)=\bigcup_\gamma C_\gamma$.
	The maps $\tilde p_\gamma^\vee:P(i_\circ C_\gamma)\to D_\gamma$ are
	generically radicial.
	For $k$ perfect the maps $\tilde p_\gamma^{\vee}$ are birational
	unless $\charac k=2$ in which case the generic degree may also be 2.
\end{proposition-cite*}
Beilinson's proof of this theorem relies on a delicate analysis of the
geometry of the Veronese embedding, and holds verbatim in our setting.
We simply remark that if $\mathcal G$ is an irreducible perverse sheaf
supported on $\tilde\PP^\vee$ which is not lisse, and
$U\subset\tilde\PP^\vee$ is the maximal open subset of $\tilde\PP^\vee$
such that $\mathcal G|_U$ is lisse, then $\tilde\PP^\vee\smallsetminus U$
is a divisor. For this, we may assume $\Lambda$ is a finite extension $E$
of $\QQ_\ell$, and we argue as follows.
We may find a lattice for the $\pi_1(U)$-representation corresponding to
$\mathcal G$, hence a local system $\mathcal L_0$ on $U$ such that
$\mathcal G=j_{!*}(\mathcal L[n])$, where 
$j:U\hookrightarrow\tilde\PP^\vee$ and
$\mathcal L=\mathcal L_0[\ell^{-1}]$.
Let $^\circ j_*:=\mathcal H^0j_*$ denote the un-derived version of
the direct image $D(U)^\heartsuit\to D(\tilde\PP^\vee)^\heartsuit$ (hearts 
taken with respect to the usual $p=0$ t-structure).
Let $\mathcal G_0:=j_{!*}(\mathcal L_0[n])$. Then $\mathcal G_0$ is an
integral model for $\mathcal G$ and has
$\mathcal G_0|_U\simeq\mathcal L_0[n]$.
Then the maximal open subset $V_0\subset\tilde\PP^\vee$ on which 
$\mathcal G_0=j_{!*}(\mathcal G_0|_U)$ is lisse is
$\bigcap_nV_n$, where $V_n\subset\tilde\PP^\vee$ is the maximal open subset
such that $^\circ j_*\mathcal L_n$ is (\'etale) locally constant.
Indeed, if $L_0$ is the representation of
$\Gal(\sep{(\Frac\tilde\PP^\vee)}/\Frac\tilde\PP^\vee)$ corresponding to
$\mathcal L_0$, then the largest open set $V'_0\subset\tilde\PP^\vee$ to
which $L_0$ extends as a representation of $\pi_1(V_0')$ is $\bigcap_nV_n$,
and if $\mathcal L_0'$ is the local system on $V_0'$ with 
$\mathcal L_0'|_U=\mathcal L_0$, and
$\kappa:U\hookrightarrow V_0'$ denotes the open immersion, then 
$\kappa_{!*}(\mathcal L_0[n])
=({^\circ\kappa_*}\mathcal L_0)[n]=\mathcal L_0'[n]$,
which shows $V_0=V_0'$.
Zariski-Nagata purity shows that each $\tilde\PP^\vee\smallsetminus V_n$ is 
a divisor. Thus $\tilde\PP^\vee\smallsetminus V_0$ is a divisor as it is a
union of divisors.

Now, if $\tilde\PP^\vee\smallsetminus U$ were not a divisor, by the above
discussion for $\mathcal G_0$, we would find that $\mathcal G_0$ is lisse
on a strictly larger $U'\supsetneq U$.
But then $\mathcal G=\mathcal G_0[\ell^{-1}]$ must also be lisse on $U'$, a
contradiction.
This proves Theorem~\ref{th:SS}\ref{th:SS:dim}.

\subsection{}\label{sec:SSw}
We turn to the proof of the equality $SS=SS^w$, the inclusion
$SS^w\subset SS$ being automatic. Beilinson's proof \cite[\S4.9]{Sasha}
continues to hold in our setting, and we only need to say a few words about
the specialization argument that concludes his proof.
We are in the setting of $\mathcal F\in D(\PP)$, and we've chosen an
irreducible component $C_\gamma$ of $SS(\mathcal F)$. Then 
$D_\gamma:=D_{C_\gamma}$ is an irreducible component of the divisor 
$D_{\mathcal F}$, and we choose an open dense subset 
$D_{\gamma}^o\subset D_\gamma$ with the properties (in particular) that
$D_\gamma^o$ doesn't intersect other components of $D$ and that
$\tilde R(i_*\mathcal F)|_{D_\gamma^o}$ is lisse.
Next, we choose a line $L\subset\tilde\PP^\vee$ with the property
(in particular) that $L$ intersects $D$ properly and 
$L\cap D_\gamma\subset D_\gamma^o$.
Picking a geometric point $\tilde x^\vee\to L\cap D_\gamma$, the proof
rests on the fact that the vanishing cycles of $\tilde R(i_*\mathcal F)|_L$
at $x^\vee$ are nonzero (with respect to $\id:L\to L$).
This follows from the following
\begin{lemma*}
Let $\mathcal F$ be a constructible sheaf on an irreducible scheme
$X=\overline{\{\eta\}}$, and let $U$ be the largest open subset of $X$
where $\mathcal F$ is lisse. Let $\xi$ be a (geometric point centered on a)
generic point of $X\smallsetminus U$. Suppose $X\smallsetminus U$ has 
locally finitely many irreducible components. Then no specialization
morphism from a point of $U$ to $\xi$ induces an isomorphism on stalks of
$\mathcal F$.
\end{lemma*}
(Compare \cite[Exp. IX Proposition 2.11]{SGAA}.)
Let's admit the lemma and see how it implies the desired conclusion.
If $\mathcal G$ denotes $\tilde R(i_*\mathcal F)$, then $\mathcal G$ is
lisse on the complement of $D$ but not on a neighborhood of the generic
point of $D_\gamma^o$. It's also lisse when restricted to $D_\gamma^o$ by
construction.
Let $\xi$ be a geometric point centered on the generic point of 
$D^o_\gamma$ and fix specialization morphisms
$\xi\rightsquigarrow\tilde x^\vee$ and $\overline\eta\rightsquigarrow\xi$.
Recalling that being lisse is equivalent to being universally locally 
acyclic with respect to the identity map by 
Lemma~\ref{lem:id_lisse}\ref{lem:id_lisse:id}, we have that
$\mathcal F_{\tilde x^\vee}\xra\sim\mathcal F_\xi$ is an isomorphism and
that any specialization map between points of
$\PP\smallsetminus D_{\mathcal F}$ induce isomorphisms on stalks of
$\mathcal F$. However, $\mathcal F_\xi\to\mathcal F_{\overline\eta}$ is not
an isomorphism by the lemma.
Therefore the composite specialization morphism 
$\mathcal F_{\tilde x^\vee}\to\mathcal F_{\overline\eta}$ is not an
isomorphism. Let $\zeta$ denote the generic point of $L$; it belongs to
$\PP\smallsetminus D_{\mathcal F}$, so any specialization 
$\overline\eta\rightsquigarrow\zeta$ induces an isomorphism
$\mathcal F_\zeta\xra\sim\mathcal F_{\overline\eta}$.
It follows that, fixing a specialization map
$\zeta\rightsquigarrow\tilde x^\vee$,
$\mathcal F_{\tilde x^\vee}\to\mathcal F_\zeta$ is not an isomorphism,
hence that the vanishing cycles for $\id:L\to L$ are nonzero at
$\tilde x^\vee$, as desired.
\begin{proof}[Proof of lemma]
	If any one specialization morphism from a point of $U$ to $\xi$ induces 
	an isomorphism on stalks of $\mathcal F$, then every specialization
	morphism from a point of $U$ to $\xi$ does, as any two points of $U$ are
	specializations of $\eta$ and $\mathcal F$ is lisse on $U$.
	If this is the case, then there is a neighborhood of $\xi$ in $X$ where
	all specialization morphisms induce isomorphisms on stalks of 
	$\mathcal F$. (Indeed, possibly replacing $X$ by an affine neighborhood
	of $\xi$, given a partition of $X\smallsetminus U$
	witnessing the constructibility of $\mathcal F$, the closure of the
	union of the strata not containing $\xi$ does not contain $\xi$.)
	Then $\mathcal F$ is universally locally acyclic with respect to the
	identity morphism on that larger open subset, hence lisse there, 
	contradicting the maximality of $U$.
\end{proof}
\begin{remark*}
	The definition of universal local acyclicity requires that we check that
	the specialization maps on our neighborhood of $\xi$ continue to induce
	isomorphisms on stalks of $\mathcal F$ after arbitrary base change in
	$X$, but this follows from what we've already checked.
	Indeed, let $f:Y\to X$ be any morphism of schemes and let 
	$y_1\rightsquigarrow y_0$ be a specialization of geometric points in 
	$Y$. When $\mathcal F$ is an abelian \'etale sheaf, it's easy to see
	that $\operatorname{sp}:f^*(\mathcal F)_{y_0}\to f^*(\mathcal F)_{y_1}$
	coincides with the map
	$\operatorname{sp}:\mathcal F_{f(y_0)}\to\mathcal F_{f(y_1)}$
	($f(y_1)$ is a geometric point of $X_{f(y_0)}$ as $y_1$ is a geometric
	point of $Y_{y_0}$). This implies the same for constructible sheaves 
	with $\ell$-adic coefficients.
\end{remark*}
This proves Theorem~\ref{th:SS}\ref{th:SS:SSw}, and
Theorem~\ref{th:SS}\ref{th:SS:verdier} follows by the same argument as in
\S\ref{sec:integral_SS}. Theorem~\ref{th:SS}\ref{th:SS:constituents}
follows from the equality $SS=SS^w$ in light of the perverse t-exactness of
the vanishing cycles functors $\phi_f$.
Indeed, if $\{\mathcal F_\alpha\}$ are the constituents of some 
$\mathcal F\in D(X)$, then 
$SS(\mathcal F)\subset\bigcup_\alpha SS(\mathcal F_\alpha)$ by 
Lemma~\ref{lem:B2.1}\ref{lem:B2.1:thick}. For the reverse inclusion, assume
$k$ is infinite, suppose $0\to\mathcal K\to\mathcal M\to\mathcal N\to0$ is
an exact sequence of perverse sheaves on $X$, and fix a geometric point
$x\to X$ and a function $f:X\to\A^1$ so that $f(x)=0$ and such that
$\phi_f(\mathcal N)_x\ne0$. Then $x$ must belong to
$\supp\phi_f(\mathcal M)$; if it didn't, we could replace $X$ by an open
neighborhood of $x$ and find that $\phi_f(\mathcal N)_x=0$ by the 
t-exactness of $\phi_f$. If $x\in\supp\phi_f(\mathcal M)$, then
$df(y)\in SS^w(\mathcal M)$ for all $y$ in a locally closed subset
$V$ of the geometric special fiber with $x\in\overline V$.
The description of $SS^w(\mathcal M)$ in \S\ref{sec:ss_def} then implies
that $df(x)\in SS^w(\mathcal M)$. Of course the same argument holds with
$\mathcal K$ in lieu of $\mathcal N$, and shows that 
$SS^w(\mathcal M)=SS^w(\mathcal K)\cup SS^w(\mathcal N)$.
The same argument applied to
$^p\tau^{\leq 0}\mathcal F\to\mathcal F\to{^p\tau^{>0}}\mathcal F\to\ $
shows that $SS^w(\mathcal F)$ contains both
$SS^w(^p\tau^{\leq 0}\mathcal F)$ 
and $SS^w(^p\tau^{>0}\mathcal F)$, so
$SS^w(\mathcal F)=\bigcup_i SS^w(^p\mathcal H^i\mathcal F)$. This part
works just as well with integral coefficients, which proves the first
sentence of Theorem~\ref{th:SS}\ref{th:SS:constituents} in general.
(The case of finite $k$ reduces to this one after making a finite extension
of the base field and using Remark~\ref{remark:SS_finite_extn}.)

Theorem~\ref{th:SS}\ref{th:SS:smooth_pullback} holds with the original 
proof \cite[\S4.10]{Sasha}, relying on the commutation of $\phi_f$ with
smooth pullback. 

Finally, to prove Theorem~\ref{th:SS}\ref{th:SS:field_extn}, by
Proposition~\ref{prop:B1.7} it suffices to show that if $\mathcal G$ is a
sheaf on $\tilde\PP^\vee$, $U\subset\tilde\PP^\vee$ is the maximal open
subset where $\mathcal G$ is lisse with complement 
$D:=\tilde\PP^\vee\smallsetminus U$, and $k'/k$ is any field extension,
then $U_{k'}$ is the maximal open subset of $\tilde\PP^\vee_{k'}$ where 
$\mathcal G_{k'}$ is lisse.
Clearly, $\mathcal G_{k'}$ is lisse on $U_{k'}$. If $\mathcal G_{k'}$ is
lisse on some strictly larger open subset, then it's lisse on a 
neighborhood of a generic point of $D_{k'}$.
Let $x_0$ be a geometric point centered on that generic point, defining 
also a geometric generic point of $D$, and let $x_1$ be a geometric point
of $U_{k'}$ specializing to $x_0$. Then the specialization map
$\mathcal G_{x_0}\to\mathcal G_{x_1}$ is an isomorphism.
But by the remark above, this specialization map coincides with the
corresponding specialization map in $\tilde\PP^\vee$, which can't be an
isomorphism by the previous lemma, since $\mathcal G$ isn't lisse on a
neighborhood of $\xi\in D$ by hypothesis.


This completes the proof of Theorem~\ref{th:SS} with rational coefficients.

\end{document}